\PassOptionsToPackage{giveninits=true}{biblatex}
\documentclass[arxiv=true]{agn_article}
\usepackage{agn_all}
\usepackage{multicol}
\usepackage{hyperref}
\usepackage{caption}
\newcommand{\sample}{343}
\newcommand{\graphRatio}{.63}

\DeclareMathOperator{\CSO}{CSO}
\DeclareMathOperator{\arccosh}{arccosh}
\DeclareMathOperator{\conv}{conv}
\newcommand{\Cm}{C_m}
\newcommand{\WdH}{W_{\mathrm{dH}}}
\newcommand{\Wlog}{W_{\mathrm{log}}}
\newcommand{\opnorm}[1]{\norm{#1}_{\mathrm{op}}}
\newcommand{\de}{\mathrm{dist}_{\mathrm{euclid}}}
\DeclareMathOperator*{\esssup}{ess\,sup}
\newcommand{\citet}[2][]{\citeauthor{#2} \cite[#1]{#2}}
\graphicspath{{gfx/}}
\begin{document}
\title{The quasiconvex envelope of conformally invariant planar energy functions in isotropic hyperelasticity}
\date{\today}
\author{%
	Robert J. Martin%
	\thanks{\;\;%
		Corresponding author: Robert J. Martin, Lehrstuhl f\"{u}r Nichtlineare Analysis und Modellierung, Fakult\"at f\"ur Mathematik, Universit\"at Duisburg--Essen, Campus Essen, Thea-Leymann Stra{\ss}e 9, 45141 Essen, Germany,
		email: \href{mailto:robert.martin@uni-due.de}{robert.martin@uni-due.de}%
	},
	\quad
	Jendrik Voss\thanks{\;\;%
		Jendrik Voss, Lehrstuhl f\"{u}r Nichtlineare Analysis und Modellierung, Fakult\"{a}t f\"{u}r Mathematik, Universit\"{a}t Duisburg--Essen, Thea-Leymann Str. 9, 45127 Essen, Germany;
		email: \href{mailto:max.voss@uni-due.de}{max.voss@uni-due.de}%
	},
	\quad
	Ionel-Dumitrel Ghiba\thanks{\;\;%
		Ionel-Dumitrel Ghiba, Alexandru Ioan Cuza University of Ia\c si, Department of Mathematics, Blvd.~Carol I, no.\,11, 700506 Ia\c si, Romania; and Octav Mayer Institute of Mathematics of the Romanian Academy, Ia\c si Branch, 700505 Ia\c si,
		email: \href{mailto:dumitrel.ghiba@uaic.ro}{dumitrel.ghiba@uaic.ro}%
	},
	\\[.21em]
	Oliver Sander\thanks{\;\;%
		Oliver Sander, Institut für Numerische Mathematik, Technische Universität Dresden, Zellescher Weg 12--14, 01069 Dresden, Germany,
		email: \href{mailto:oliver.sander@tu-dresden.de}{oliver.sander@tu-dresden.de}%
	}
	~\quad and\quad%
	Patrizio Neff\thanks{\;\;%
		Patrizio Neff, Head of Lehrstuhl f\"{u}r Nichtlineare Analysis und Modellierung, Fakult\"at f\"ur Mathematik, Universit\"at Duisburg--Essen, Campus Essen, Thea-Leymann Stra{\ss}e 9, 45141 Essen, Germany,
		email: \href{mailto:patrizio.neff@uni-due.de}{patrizio.neff@uni-due.de}%
	}
}
\maketitle
\vspace*{-1.4em}
\begin{center}
	Dedicated to Philippe G.~Ciarlet on the occasion of his 80$^\text{th}$ birthday.
\end{center}
\vspace*{1.4em}
\begin{abstract}
	\noindent
	We consider \emph{conformally invariant} energies $W$ on the group $\GLp(2)$ of $2\times2$-matrices with positive determinant, i.e.\ $W\col\GLp(2)\to\R$ such that
	\[
		W(A\.F\.B) = W(F) \qquad\text{for all }\; A,B\in\{a\.R\in\GLp(2) \setvert a\in(0,\infty)\,,\; R\in\SO(2)\}\,,
	\]
	where $\SO(2)$ denotes the special orthogonal group, and provide an explicit formula for the (notoriously difficult to compute) \emph{quasiconvex envelope} of these functions. Our results, which are based on the representation $W(F)=h(\frac{\lambda_1}{\lambda_2})$ of $W$ in terms of the singular values $\lambda_1,\lambda_2$ of $F$, are applied to a number of example energies in order to demonstrate the convenience of the eigenvalue-based expression compared to the more common representation in terms of the distortion $\K\colonequals\frac12\frac{\norm{F}^2}{\det F}$. Special cases of our results can be obtained from earlier works by Astala et al.\ \cite{astala2008elliptic} and Yan \cite{yan2003baire}.
\end{abstract}

\vspace*{1.4em}
\\[1.4em]
\newpage
\tableofcontents

\section{Introduction}
A recent contribution \cite{agn_martin2015rank} introduced a number of criteria for generalized convexity properties (including quasiconvexity) of so-called \emph{conformally invariant functions} (or \emph{energies}) on the group $\GLp(2)$ of $2\times2$-matrices with positive determinant, i.e.\ functions $W\col\GLp(2)\to\R$ with
\begin{equation}\label{eq:definitionConformalInvariance}
	W(A\.F\.B) = W(F) \qquad\text{for all }\; A,B\in\CSO(2)\,,
\end{equation}
where
\[
	\CSO(2)\colonequals\R^+\cdot\SO(2)=\{a\.R\in\GLp(2) \setvert a\in(0,\infty)\,,\; R\in\SO(2)\}
\]
denotes the \emph{conformal special orthogonal group}.\footnote{%
\label{footnote:conformalEnergy}
	Note that this invariance property needs to be distinguished from the concept of (nearly) \emph{conformal energies} \cite{iwaniec1996polyconvex,yan1997rank}, i.e.\ functions $W\geq 0$ such that $W(F)=0$ if and only if $F\in\CSO(2)$, e.g.\ $W(F)=\norm{F}^2-2\det F$. Instead of invariances of the argument, these energies are characterized by a global \enquote{potential well} containing the unbounded set $\CSO(2)$ and can merely be considered \enquote{conformally invariant in $F=\id$}.
	
	In a planar minimization problem subject to the homeomorphic boundary condition $\varphi|_{\partial\Omega}=\varphi_0$, the 2-harmonic Dirichlet energy $I(\varphi)=\int_{\Omega}\frac{1}{2}\.\norm{\nabla\varphi}^2\.\dx$ is sometimes referred to as a conformal energy as well. Indeed,
	\[
		I(\varphi)=\int_{\Omega}\frac{1}{2}\.\norm{\nabla\varphi}^2\,\dx\geq \int_{\Omega}\det\nabla\varphi(x)\,\dx = \int_{\Omega}\det \nabla\varphi_0\,\dx\,,
	\] 
	and equality holds
	if and only if $\varphi$ is conformal, due to Hadamard's inequality and the fact that $\det\nabla\varphi$ is a null Lagrangian. However, the energy density $W(F)=\frac{1}{2}\.\norm{F}^2$ is neither conformally invariant in the sense of \eqref{eq:definitionConformalInvariance} nor (nearly) conformal in the above sense.%
}
This requirement can equivalently be expressed as
\begin{equation}\label{eq:objectiveIsotropicIsochoric}
	W(R_1F)=W(F)=W(FR_2)\,,\quad W(aF)=W(F) \qquad\text{for all }\; R_1,R_2\in\SO(2)\,,\; a\in(0,\infty)\,,
\end{equation}
i.e.\ left- and right-invariance under the special orthogonal group $\SO(2)$ and invariance under scaling. In nonlinear elasticity theory, where $F=\grad\varphi$ represents the so-called deformation gradient of a deformation~$\varphi$, the former two invariances correspond to the \emph{objectivity} and \emph{isotropy} of $W$, respectively. In this context, an energy $W$ satisfying $W(aF)=W(F)$ is more commonly known as \emph{isochoric}, and is often additively coupled \cite{richter1949verzerrung,charrier1988existence} with a \emph{volumetric} energy term of the form $f(\det F)$ for some convex function $f\col(0,\infty)\to\R$.

In this contribution, we consider the \emph{quasiconvex envelopes} of conformally invariant energies on $\GLp(2)$. Based on our previous results, we provide an explicit formula that allows for a direct computation of the quasiconvex (as well as the rank-one convex and polyconvex) envelope for this class of functions. We also discuss different ways of expressing conformally invariant energies, including representations based on the singular values of $F$, i.e.\ the eigenvalues of $\sqrt{F^TF}$, in order to highlight the difficulties which arise from focusing on the seemingly more simple representation in terms of the distortion function $\K$.

Our main result (Theorem~\ref{theorem:mainResult}) has been tested against a numerical algorithm for computing the polyconvex envelope \cite{bartels2005reliable} for a range of parameters, yielding agreement up to computational precision. In two special cases, we show that our results completely match previous developments of \citet{astala2008elliptic} and Yan \cite{yan2003baire,yan2001linear}. We also present direct finite element simulations of the microstructure using a trust-region--multigrid method~\cite{conn2000trustregion,sander2012simplicial} which show consistent results.
In the appendix, we answer two questions by \citet{adamowicz2007grotzsch}, and discuss a related relaxation result by \citet{dacorogna1993different}.

\subsection{Conformal and quasiconformal mappings}
Energy functions of the form \eqref{eq:definitionConformalInvariance} are intrinsically linked to conformal geometry and geometric function theory \cite{astala2008elliptic}. A mapping $\varphi\col\Omega\to\R^2$ is called \emph{conformal} if and only if $\grad\varphi(x)\in\CSO(2)$ on $\Omega$ or, equivalently,
\begin{align*}
	\grad\varphi^T\grad\varphi = (\det\grad\varphi)\cdot\id\,,%
\end{align*}
where $\id\in\GLp(2)$ denotes the identity matrix. If $\R^2$ is identified with the complex plane $\C$, then $\varphi$ is conformal if and only if $\varphi\col\Omega\subset\C\to\C$ is holomorphic and the derivative is non-zero everywhere.
Although the Riemann mapping theorem states that any non-empty, simply connected open planar domain can be mapped conformally to the unit disc, conformal mappings exhibit aspects of rigidity \cite{faraco2005geometric} that make them too restrictive for many interesting applications. In particular, since the Riemann mapping is uniquely determined by prescribing the function value for three points, conformal mappings are not able to satisfy arbitrary boundary conditions.

A significantly larger and more flexible class is given by the so-called \emph{quasiconformal} mappings, i.e.\ functions $\varphi\col\Omega\to\R^2$ that satisfy the uniform bound
\begin{equation}\label{eq:quasiConformalBound}
	\norm{\K}_\infty\colonequals\esssup_{x\in\Omega}\K(\grad\varphi(x))\leq L \qquad\text{for some }\;L\geq1\,,
\end{equation}
where $\K$ denotes the \emph{distortion function} \cite{iwaniec2009,astala2010deformations} or \emph{outer distortion} \cite{iwaniec2011invitation}
\begin{equation}\label{eq:distortionFunction}
	\K\col\GLp(2)\to\R\,,\quad \K(F) \colonequals \frac12\.\frac{\norm{F}^2}{\det F} = \frac{\sum_{i,j=1}^2 F_{ij}^2}{2\.\det F}\,.
\end{equation}
Due to Hadamard's inequality, $\K(F)\geq 1$ for all $F\in\GLp(2)$. In particular, if \eqref{eq:quasiConformalBound} is satisfied with $L=1$, then $\K(\grad\varphi)\equiv1$, which implies that $\varphi$ is conformal.

The classical \emph{Grötzsch free boundary value problem} \cite{grotzsch1928einige} (cf.\ Appendix~\ref{section:groetzsch}) is to find and characterize quasiconformal mappings of rectangles into rectangles that minimize the maximal distortion $\norm{\K}_\infty$ and map faces to corresponding faces, i.e.\ to solve the minimization problem
\begin{alignat}{2}
	\norm{\K(\grad\varphi)}_\infty \to \min\,,\qquad
	\varphi\col[0,a_1]\times[0,1]&\to[0,a_2]\times\mathrlap{[0,1]\,,}&&\notag\\
	\varphi([0,a_1]\times\{0\})&=[0,a_2]\times\{0\}\,,&
	\;\varphi([0,a_1]\times\{1\})&=[0,a_2]\times\{1\}\,,\\
	\varphi(\{0\}\times[0,1])&=\{0\}\times[0,1]\,,&
	\;\varphi(\{a_1\}\times[0,1])&=\{a_2\}\times[0,1]
	\,.\notag
\end{alignat}
A much more involved problem has been solved by Teichmüller \cite{teichmuller1944verschiebungssatz,alberge2015commentary}. The classical Teichmüller problem (cf.\ Appendix~\ref{section:groetzsch}) is to find and characterize quasiconformal solutions to
\begin{align}
	\norm{\K(\grad\varphi)}_\infty \to \min\,,\qquad \varphi\in W^{1,2}(\Omega;\R^2)\,,\quad\varphi(x)|_{\partial B_1(0)}=x\,,\quad\varphi(0)=-a\label{eq:teichmueller}
\end{align}
for $0<a<1$.
Computational approaches for calculating extremal quasiconformal mappings (with direct applications in engineering) are discussed, e.g., in \cite{weber2012computing}. However, the analytical difficulties posed by this problem also motivate the study of integral generalizations of \eqref{eq:teichmueller}, i.e.
\begin{align*}
	\int_{B_1(0)}\Psi(\K(\nabla\varphi))\,\dx \to \min\,,\qquad
	\varphi\in W^{1,2}(\Omega;\R^2)\,,\quad\varphi(x)|_{\partial B_1(0)}=x\,,\quad\varphi(0)=-a\,,
\end{align*}
where $\Psi\col[1,\infty)\to[0,\infty)$ is assumed to be strictly increasing.
Further generalizing the domain, boundary condition and additional constraints, we obtain a more classical problem in the calculus of variations: %
the existence and uniqueness of mappings between planar domains with prescribed boundary values that minimize certain integral functions of $\K$, i.e.\ the minimization problem
\begin{equation}\label{eq:conformalInfProblem}
	\int_\Omega \Psi(\K(\grad\varphi))\,\dx \;\to\min\,,\qquad \varphi\in W^{1,2}(\Omega;\R^2)\,,\quad \varphi\big|_{\partial\Omega} = \varphi_0\big|_{\partial\Omega}
\end{equation}
for given $\Psi\col[1,\infty)\to\R$ and $\varphi_0\col\Omega\to\R^2$. Since $\K(aR\.\grad\varphi)=\K(\grad\varphi\.aR)=\K(\grad\varphi)$ for all $a>0$ and all $R\in\SO(2)$, the distortion function $\K$ is conformally invariant, and indeed every conformally invariant energy $W$ on $\GLp(2)$ can be expressed in the form $W(F)=\Psi(\K(F))$, see \cite{agn_martin2015rank}.

However, the mapping $F\mapsto\K(F)$ is non-convex. Without additional restrictions on $\Psi$, it is therefore difficult to establish results regarding the existence or regularity of minimizers. It is generally believed \cite[Conjecture 21.2.1, p.\ 599]{astala2008elliptic} that for \enquote{well-behaved} functions $\Psi$, e.g.\ if $\Psi$ is smooth, strictly increasing and convex, any solution\footnote{%
	Note that the existence of minimizers follows from the polyconvexity \cite{Dacorogna08,charrier1988existence,ball1976convexity} of the mapping $F\to\Psi(\K(F))$.}
to the minimization problem \eqref{eq:conformalInfProblem} is a $C^{1,\alpha}$-diffeomorphism; this would contrast typical regularity results for more general problems in the calculus of variations (including nonlinear elasticity), where only partial regularity (e.g.\ $C^{1,\alpha}$ up to a set of measure zero) can be expected.

In this contribution, we are interested in cases where $\Psi$ is \emph{not} well behaved in the above sense; more specifically, we allow for some lack of convexity \emph{and} monotonicity of $\Psi$. Our results demonstrate that the common representation $W(F)=\Psi(\K(F))$ of an arbitrary conformally invariant function $W$ on $\GLp(2)$ is neither ideal nor \enquote{natural} as far as convexity properties of $W$ are concerned. Instead, by introducing the \emph{linear distortion} (or \emph{(large) dilatation} \cite{weber2012computing})
\[
	K(F) = \frac{\opnorm{F}^2}{\det F}=\frac{\lambdamax(\sqrt{F^TF})}{\lambdamin(\sqrt{F^TF})}=\K(F)+\sqrt{\K(F)^2-1}=e^{\arccosh\K(F)}\,,\qquad\text{i.e.}\quad
	\K=\frac{1}{2}\left(K+\frac{1}{K}\right)\,,%
\]
where $\opnorm{F}=\sup_{\norm{\xi}=1}\norm{F\.\xi}_{\R^2}$ denotes the operator norm (i.e.\ the largest singular value) of $F$, we can equivalently express any conformally invariant energy $W$ as $W(F)=h(K(F))$. This representation turns out to be much more convenient and suitable with respect to convexity properties of $W$.\footnote{On the other hand, since the distortion function $\K$ is differentiable on all of $\GLp(2)$ whereas the mapping $F\mapsto K(F)$ is not, it is advised to base numerical approaches to relaxation of conformally invariant energies on $\K$.}

In particular, our results (cf.\ Remark \ref{remark:iwaniecResult}) will allow us to easily generalize a consequence of a theorem by Astala, Iwaniec and Martin \cite[Theorem 21.1.3, p.\ 591]{astala2008elliptic}, stating that for $F_0\in\GLp(2)$ and $\Omega=B_1(0)$ and any strictly increasing $\Psi\col[1,\infty)\to[0,\infty)$ with sublinear growth,
\begin{equation}\label{eq:iwaniecResult}
	\inf \left\{ \int_{B_1(0)} \Psi(\K(\grad\varphi))\,\dx\,,\; \varphi\in W^{1,2}(B_1(0),\R^2)\,,\; \varphi\big|_{\partial B_1(0)}(x)=F_0\.x \right\} \;=\; \pi\cdot\Psi(1)\,.
\end{equation}
Note that the corresponding minimization problem has no solution unless $F_0\in\CSO(2)$. This result can be expressed as a specific \emph{relaxation} statement, namely that for these $W(F)=\Psi(\K(F))$, the quasiconvex envelope $QW$ of $W$ is given by $QW(F)\equiv\Psi(1)$. Quasiconvex envelopes arise naturally in the calculus of variations, representing the energetic response of a minimization problem without minimizers for linearly homogeneous boundary values under the presence of microstructure, i.e.\ \cite{Dacorogna08,vsilhavy2001rank,pedregal2000variational,silhavy1997mechanics}
\begin{equation}
	QW(F_0) = \inf \left\{ \frac{1}{\abs{\Omega}}\int_\Omega W(\grad\varphi)\,\dx\,,\; \varphi\in W^{1,p}(\Omega,\R^2)\,,\; \varphi\big|_{\partial\Omega}(x)=F_0\.x \right\}
\end{equation}
for any domain $\Omega\subset\R^2$ with Lebesgue measure $\abs{\Omega}$. If $QW(F_0)<W(F_0)$ for some $F_0\in\GLp(2)$, then the equilibrium state of the homogeneous deformation $\varphi(x)=F_0\.x$ is unstable: the material shows an energetic preference to develop finer and finer spatially modulated deformations (in engineering applications these are typically shear bands) at fixed averaged deformation $F_0\.x$. In this case, there are infimizing sequences with highly oscillating gradients which converge weakly (presuming appropriate coercivity conditions), but the weak limit is not a minimizer.

In this respect, our main result also answers a particular case of Iwaniec's Question~3 from \cite{iwaniec1999nonlinear}: \enquote{Which [functions of the] distortion [$\Psi(\K(\nabla\varphi))$] are weakly lower semi-continuous in $W^{1,2}(\Omega,\R^2)$?}

\section{Convexity properties of conformally invariant functions}

In (planar) nonlinear elasticity theory, three generalized convexity properties \cite{ball1976convexity,Dacorogna08} of an energy function $W\col\GLp(2)\to\R$ are of particular interest:
\emph{rank-one convexity}, i.e.\ for all $F\in\GLp(2)$, $t\in[0,1]$ and $H\in\R^{2\times2}$ with $\rank(H)=1$ and $\det(F+H)>0$,
\begin{equation*}
	W((1-t)F+t(F+H)) \;\leq\; (1-t)\.W(F) + t\.W(F+tH)\,;
\end{equation*}
\emph{quasiconvexity}, i.e.\ for every bounded open set $\Omega\subset\R^2$ and all $\vartheta\in C_0^\infty (\Omega)$ such that $\det(F_0+\nabla \vartheta)>0$,
\begin{equation*}
	\int_{\Omega}W(F_0+\nabla \vartheta)\,\dx\geq \int_{\Omega}W(F_0)\,\dx=W(F_0)\cdot \abs{\Omega}\,;
\end{equation*}
and \emph{polyconvexity} of $W$, i.e.
\begin{equation*}
	W(F) = P(F,\det F) \qquad\text{for some convex function }\;P\col\R^{2\times2}\times\R \;\cong\; \R^5\to\R\cup\{+\infty\}\,.
\end{equation*}
The importance of quasiconvexity stems from the fact that quasiconvexity of $W$ is essentially equivalent to the weak lower semi-continuity of $\int_\Omega W(\nabla\varphi(x))\.\dx$ \cite{morrey1952quasi}. The problem 
\[
	\int_\Omega W(\nabla\varphi(x))\,\dx \to \min\,,\qquad \varphi\in W^{1,2}(\Omega,\R^2)\,,\quad \varphi\big|_{\partial\Omega}(x)=F_0\.x
\]
has a solution if and only if $QW(F_0)=W(F_0)$; if $QW(F_0)<W(F_0)$, then infimizing sequences exhibit microstructure effects.

Under appropriate technical assumptions,
\begin{align}
	\text{convexity }\implies\text{ polyconvexity }\implies\text{ quasiconvexity }\implies\text{ rank-one convexity}\,,\label{eq:convexImplication}
\end{align}
whereas for dimension $n\geq3$, it is well known that the corresponding convexity properties are not equivalent; Sverak famously showed that rank-one convexity does not imply quasiconvexity with a counterexample consisting of a non-isotropic, non-objective polynomial of order four \cite{vsverak1992rank}. In the two-dimensional case discussed here, however, the question whether rank-one convexity is equivalent to quasiconvexity, known as the remaining part of \emph{Morrey's conjecture} \cite{morrey1952quasi}, is still unanswered \cite{morrey1952quasi,astala2012quasiconformal} and is considered one of the major open problems in the calculus of variations \cite{ball1987does,ball2002openProblems,neff2005critique}.

\medskip
\noindent In order to state criteria for the above convexity properties in the special case of conformally invariant functions on $\GLp(2)$, we consider a number of different representations available to express such functions.
\pagebreak
\begin{lemma}[\cite{agn_martin2015rank}]
\label{lemma:representations}
	Let $W\col\GLp(2)\to\R$ be conformally invariant. Then there exist uniquely determined functions $g\col(0,\infty)\times(0,\infty)\to\R$,\; $h\col(0,\infty)\to\R$ and $\Psi\col[1,\infty)\to\R$ such that
	\begin{equation}\label{eq:representationFormulae}
		W(F) = g(\lambda_1,\lambda_2)
		= h\left(\frac{\lambda_1}{\lambda_2}\right)
		= \Psi(\K(F))
	\end{equation}
	for all $F\in\GLp(2)$ with singular values $\lambda_1,\lambda_2$, where $\K(F)\colonequals\frac12\.\frac{\norm{F}^2}{\det F}$ and $\norm{\,.\,}$ denotes the Frobenius matrix norm with $\norm{F}^2=\sum_{i,j=1}^2 F_{ij}^2$. Furthermore,
	\begin{equation}\label{eq:representationRequirements}
		h(x)=h\left(\frac1x\right)\,,\quad g(x,y)=g(y,x) \quad\text{ and }\quad g(ax,ay)=g(x,y)
	\end{equation}
	for all $a,x,y\in(0,\infty)$.
\end{lemma}

Conversely, if the requirements \eqref{eq:representationRequirements} are satisfied for otherwise arbitrary functions $g\col(0,\infty)\times(0,\infty)\to\R$,\; $h\col(0,\infty)\to\R$ or $\Psi\col[1,\infty)\to\R$, then \eqref{eq:representationFormulae} defines a conformally invariant function $W$.

Note that $h$ is already uniquely determined by its values on $[1,\infty)$ and recall that $\K\geq1$, with $\K(\nabla\varphi)=1$ if and only if $\varphi$ is conformal.

The following proposition summarizes the main results from \cite{agn_martin2015rank} and completely characterizes the generalized convexity of conformally invariant functions on $\GLp(2)$.
\begin{proposition}[\cite{agn_martin2015rank}]
\label{proposition:convexityCharacterization}
	Let $W\col\GLp(2)\to\R$ be conformally invariant,
	and let $g\col(0,\infty)\times(0,\infty)\to\R$,\; $h\col(0,\infty)\to\R$ and $\Psi\col[1,\infty)\to\R$ denote the uniquely determined functions with
	\[
		W(F) = g(\lambda_1,\lambda_2)
		= h\left(\frac{\lambda_1}{\lambda_2}\right)
		= \Psi(\K(F))
	\]
	for all $F\in\GLp(2)$ with singular values $\lambda_1,\lambda_2$, where $\K(F)=\frac12\frac{\norm{F}^2}{\det F}$. Then the following are equivalent:

	\setlength\columnsep{-10.5em}
	\begin{multicols}{2}
		\begin{itemize}
			\item[i)] $W$ is polyconvex,
			\item[ii)] $W$ is quasiconvex,
			\item[iii)] $W$ is rank-one convex,
			\item[iv)] $g$ is separately convex,
			\item[v)] $h$ is convex on $(0,\infty)$,
			\item[vi)] $h$ is convex and non-decreasing on $[1,\infty)$.
		\end{itemize}
	\end{multicols}
	
	\noindent
	Furthermore, if $\Psi$ is twice continuously differentiable, then i)--vi) are equivalent to
	\begin{itemize}
		\item[vii)] $(x^2-1)\,(x+\sqrt{x^2-1})\,\Psi^{\prime\prime}(x) + \Psi^{\prime}(x)\geq 0$\quad for all\; $x\in(1,\infty)$.
	\end{itemize}
\end{proposition}

Note that in terms of the representation function $h$, the convexity criteria can be expressed in a remarkably simple way, especially when compared to vii), i.e.\ the representation in terms of the classical distortion $\K$. In particular, while monotonicity and convexity of $\Psi$ are sufficient for the considered properties,\footnote{Recall that the mapping $F\mapsto\K(F)$ itself is polyconvex \cite{Dacorogna08,agn_hartmann2003polyconvexity} on $\GLp(2)$.} convexity of the energy with respect to $\K$ is not a necessary condition; for example, if $W\col\GLp(2)\to\R$ is given by
\begin{equation*}
	W(F)=\frac{\max\{\lambda_1,\lambda_2\}}{\min\{\lambda_1,\lambda_2\}} = \max\left\{\frac{\lambda_1}{\lambda_2},\frac{\lambda_2}{\lambda_1}\right\} = \K(F)+\sqrt{\K(F)^2-1} = e^{\arccosh(\K(F))}\,,
\end{equation*}
then $W$ is polyconvex due to the convexity of $x\mapsto h(x)=\max\{x,\frac1x\}$ on $(0,\infty)$, whereas the representing function $\Psi\col[1,\infty)\to\R$ with $\Psi(x)=x+\sqrt{x^2-1}$ is monotone increasing but not convex.

\begin{example}
	Consider the isochoric, conformally invariant St.\,Venant--Kirchhoff-type energy function
	\begin{equation}\label{eq:SVK}
		W\col\GLp(2)\to\R\,,\quad W(F)=\Bignorm{\frac{F^TF}{\det F}-\id}^2
		= \Big(\frac{\lambda_1}{\lambda_2}-1\Big)^2 + \Big(\frac{\lambda_2}{\lambda_1}-1\Big)^2
		= 4(\K(F)^2-\K(F))\,,
	\end{equation}
	where $\id$ denotes the identity matrix.
	This energy $W$ can be expressed in the form \eqref{eq:representationFormulae} with
	\[
		g(x,y) = \Big(\frac{x}{y}-1\Big)^2 + \Big(\frac{y}{x}-1\Big)^2\,,
		\qquad
		h(x) = (x-1)^2+\Big(\frac{1}{x}-1\Big)^2\,,
		\qquad
		\Psi(x) = 4(x^2-x)\,.
	\]
	Since $h\col(0,\infty)\to\R$ is convex, the planar isochoric St.\,Venant--Kirchhoff energy is quasiconvex according to Proposition \ref{proposition:convexityCharacterization}, while, e.g.\ the non-conformally-invariant term $\norm{F^TF-\id}^2=(\lambda_1-1)^2+(\lambda_2-1)^2$ is not, cf.\ Appendix~\ref{appendix:convexity}.
\end{example}

\section{Generalized convexity properties and convex envelopes}
\label{section:convexityProtperties}
For each of the convexity properties listed in the previous section, we can consider the corresponding \emph{envelope} of an arbitrary energy function $W\col\GLp(2)\to\R$, i.e.
\begin{align*}
	RW(F) &= \sup \{ w(F) \setvert w\col\GLp(2)\to\R \,\text{ rank-one convex, }& \hspace{-0.6cm}w(X)&\leq W(X) \,\text{ for all }X\in\GLp(2) \}\,,\\
	QW(F) &= \sup \{ w(F) \setvert w\col\GLp(2)\to\R \,\text{ quasiconvex, }& \hspace{-0.6cm}w(X)&\leq W(X) \,\text{ for all }X\in\GLp(2) \}\,,\\
	PW(F) &= \sup \{ w(F) \setvert w\col\GLp(2)\to\R \,\text{ polyconvex, }& \hspace{-0.6cm}w(X)&\leq W(X) \,\text{ for all }X\in\GLp(2) \}\,,\\
	CW(F) &= \sup \{ w(F) \setvert w\col\GLp(2)\to\R \,\text{ convex, }& \hspace{-0.6cm}w(X)&\leq W(X) \,\text{ for all }X\in\GLp(2) \}\,.
\end{align*}

Since polyconvexity implies quasiconvexity, which in turn implies rank-one convexity (cf.\ \cite[Theorem 3.3]{aubert1995necessary} for the case of isotropic functions on $\GLp(n)$), it is easy to see that $CW(F)\leq PW(F)\leq QW(F)\leq RW(F)$.\footnote{Examples of functions where $PW<QW$ were examined, for example, by Gangbo \cite{gangbo1993continuity}.} However, while a number of numerical methods are available to approximate the rank-one convex envelope $RW$ \cite{dolzmann2004variational,bartels2004linear,oberman2017partial} as well as the polyconvex envelope $PW$ \cite{dolzmann1999numerical,kruzik1998numerical,bartels2005reliable,aranda2001computation}, it is difficult to analytically compute $RW$, $PW$ or the quasiconvex envelope $QW$ of a given energy $W$ in general, although explicit representations have been found for a number of particular functions, including the St.\,Venant--Kirchhoff energy~\cite{ledret1995quasiconvex} and several challenging problems encountered in engineering applications \cite{cesana2011quasiconvex,albin2009infinite}. Further examples can be found in \cite[Chapter 6]{Dacorogna08}.

More general methods for computing the quasiconvex envelope are often based on the observation that $RW=PW$ and thus $RW=QW$ for certain classes of energy functions $W$. In many such cases, even the equality $RW=CW$ holds \cite{dacorogna1993different,raoult2010quasiconvex}, i.e.\ the generalized convex envelopes are all identical to the classical convex envelope of $W$, cf.\ Appendix~\ref{appendix:convexity}.

\citet{yan1997rank} showed that non-constant rank-one convex conformal energy functions (cf.\ Footnote~\ref{footnote:conformalEnergy} for the distinction between conformally invariant and conformal energy functions) defined on all of $\R^{n\times n}$ for $n\geq3$ must grow at least with power $\frac{n}{2}$, which implies that the quasiconvex envelope of a conformal energy $W$ on $\R^{3\times3}$ must be constant if $W$ exhibits sublinear growth.\footnote{This result is essentially sharp: Müller, {\v{S}}ver{\'a}k and Yan \cite[Theorem 1.2]{muller1999sharp} have shown that there exists a nontrivial quasiconvex conformal energy function $W\col\R^{2\times 2}\to\R$ with a constant $c^+>0$ such that for all $F\in\R^{2\times 2}$,
\[
	0\leq W(F)\leq c^+\.(1+\norm{F}) \qquad\text{and}\qquad W(F)=0 \;\iff\; F\in\CSO(2)\,.
\]}
The results given in the following show that an analogous property holds for conformally invariant energies on $\GLp(2)$.

In order to apply Proposition \ref{proposition:convexityCharacterization} to the computation of generalized convex envelopes, the following invariance property of the rank-one convex envelope will be required.
\begin{lemma}
\label{lemma:envelopeInvariance}
	If $W\col\GLp(n)\to\R$ is conformally invariant, then $RW$ is conformally invariant.
\end{lemma}
\begin{proof}
	It is well known that the left- and right-$\SO(2)$-invariance is preserved by the rank-one convex envelope \cite{buttazzo1994envelopes,dacorogna1993different,ledret1994remarks}, so due to the characterization \eqref{eq:objectiveIsotropicIsochoric} of conformal invariance it remains to show that $RW(aF)=RW(F)$ for all $a>0$ and all $F\in\GLp(2)$.
	
	We use the characterization $RW(F)=\lim_{k\to\infty}R_kW(F)$ of the rank-one convex envelope \cite[p.~202]{Dacorogna08}, where $R_0W(F)=W(F)$ and
	\[
		R_{k+1}W(F) \colonequals \inf \Big\{ t\.R_kW(F_1)+(1-t)\.R_kW(F_2) \,\setvert\, t\in[0,1],\, t\.F_1+(1-t)\.F_2=F,\, \rank(F_1-F_2)=1 \Big\}\,,
	\]
	and show by induction that $R_kW(aF)=R_kW(F)$ for all $k\geq0$. First, we find $R_0W(aF)=W(aF)=W(F)=R_0W(F)$, so assume that $R_kW(F) = R_kW(aF)$ for some $k\geq1$. For any $\eps>0$, choose $F_1,F_2\in\GLp(2)$ and $t\in[0,1]$ with $tF_1+(1-t)F_2=F$ and $\rank(F_1-F_2)=1$ such that $tR_kW(F_1)+(1-t)R_kW(F_2)\leq R_{k+1}W(F)+\eps$. Then, since $t\.aF_1+(1-t)\.aF_2=aF$ and $\rank(aF_1-aF_2)=1$,
	\begin{align*}
		R_{k+1}W(aF) \leq t\.R_kW(aF_1)+(1-t)\.R_kW(aF_2) = t\.R_kW(F_1)+(1-t)\.R_kW(F_2) \leq R_{k+1}W(F)+\eps\,,
	\end{align*}
	thus $R_{k+1}W(aF)\leq R_{k+1}W(F)$. Analogously, we find $R_{k+1}W(F)\leq R_{k+1}W(aF)$ and thereby $RW(aF)=\lim_{k\to\infty}R_kW(aF)=\lim_{k\to\infty}R_kW(F)=RW(F)$.
\end{proof}
\begin{remark}
	By direct computation, it is easy to see that $QW$ is conformally invariant if $W\col\GLp(n)\to\R$ is conformally invariant. Similar to the rank-one convex envelope it is well known \cite{buttazzo1994envelopes,dacorogna1993different,ledret1994remarks} that the left- and right-$\SO(2)$-invariance is preserved by $QW$, and the equality
	\begin{align*}
		QW(a\.F)&=\inf_{\vartheta\in\C_0^\infty(\Omega)} \frac{1}{\abs{\Omega}} \int_\Omega W(a\.F+\nabla\vartheta)\,\dx = \inf_{\vartheta\in\C_0^\infty(\Omega)} \frac{1}{\abs{\Omega}} \int_\Omega W\Big(a\Big(F+\frac{1}{a}\.\nabla\vartheta\Big)\Big)\,\dx\\
		&= \inf_{\vartheta\in\C_0^\infty(\Omega)} \frac{1}{\abs{\Omega}} \int_\Omega W\Big(F+\frac{1}{a}\.\nabla\vartheta\Big)\,\dx = \inf_{\widetilde\vartheta\in\C_0^\infty(\Omega)} \frac{1}{\abs{\Omega}} \int_\Omega W(F+\nabla\widetilde\vartheta)\,\dx=QW(F)\,,
\end{align*}
	can be obtained by utilizing the scaling invariance $W(a\.F)=W(F)$ for all $a>0$ and $F\in\GLp(n)$.
\end{remark}
\subsection{Main result on the quasiconvex envelope}
We can now state our main result.
\begin{theorem}
\label{theorem:mainResult}
	Let $W\col\GLp(2)\to\R$ be conformally invariant, and let $h\col[1,\infty)\to\R$ denote the function  uniquely determined by
	\begin{equation}\label{eq:hDefinition}
		W(F) = h\left(\frac{\lambda_1}{\lambda_2}\right)
	\end{equation}
	for all $F\in\GLp(2)$ with ordered singular values $\lambda_1\geq\lambda_2$. Then
	\begin{equation}\label{eq:mainResult}
		RW(F)=QW(F)=PW(F)=\Cm h\left(\frac{\lambda_1}{\lambda_2}\right) \qquad\text{for all }\;F\in\GLp(2)\,,
	\end{equation}
	where $\Cm h\col[1,\infty)\to\R$ denotes the \emph{monotone-convex envelope} given by
	\[
		\Cm h(t) \colonequals \sup\Big\{p(t) \setvert p\col[1,\infty)\to\R \text{ monotone increasing and convex with } p(s)\leq h(s)\;\forall\,s\in[1,\infty)\Big\}\,.
	\]
\end{theorem}
\begin{proof}
\newcommand{\mocoen}{w}
	Let $\mocoen(F)\colonequals \Cm h\left(\frac{\lambda_1}{\lambda_2}\right)$. Due to the convexity and monotonicity of $\Cm h$ and Proposition~\ref{proposition:convexityCharacterization}, the mapping $\mocoen\col\GLp(2)\to\R$ is polyconvex. Therefore, since
	\[
		\mocoen(F) = \Cm h\left(\frac{\lambda_1}{\lambda_2}\right) \leq h\left(\frac{\lambda_1}{\lambda_2}\right) = W(F)\,,
	\]
	we find $\mocoen(F)\leq PW(F)$ for all $F\in\GLp(2)$. Since $PW(F)\leq QW(F)\leq RW(F)$, it only remains to show that $RW(F)\leq\mocoen(F)$ in order to establish \eqref{eq:mainResult}.
	
	According to Lemma \ref{lemma:envelopeInvariance}, $RW$ is conformally invariant, thus according to Lemma \ref{lemma:representations} there exists a uniquely determined $\htilde\col[1,\infty)\to\R$ such that $RW(F)=\htilde\left(\frac{\lambda_1}{\lambda_2}\right)$ for all $F\in\GLp(2)$ with singular values $\lambda_1\geq\lambda_2$. Due to the rank-one convexity of $RW$ and Proposition \ref{proposition:convexityCharacterization}, the function $\htilde$ is convex and non-decreasing. Since
	\[
		\htilde(t) = RW(\diag(t,1)) \leq W(\diag(1,t)) = h(t)
	\]
	as well, we find $\htilde(t)\leq \Cm h(t)$ for all $t\in[1,\infty)$ and thus
	\[
		RW(F) = \htilde\left(\frac{\lambda_1}{\lambda_2}\right) \leq \Cm h\left(\frac{\lambda_1}{\lambda_2}\right) =\mocoen(F)
	\]
	for all $F\in\GLp(2)$.
\end{proof}

\begin{remark}
	If $h$ is monotone increasing, then $\Cm h=Ch$, i.e.\ the monotone-convex envelope (which is the largest convex non-decreasing function not exceeding $h$ \cite[Prop.\,4.1]{vsilhavy2001rank}) is identical to the (classical) convex envelope $Ch$ of $h$ on $[1,\infty)$. Furthermore, if $h$ is continuous, then computing the monotone-convex envelope $\Cm h$ can easily be reduced to the simple one-dimensional problem of finding the convex envelope $C\htilde$ of the function
	\[
		\htilde\col[1,\infty)\to\R\,,\quad \htilde(t)=
		\begin{cases}
			{\displaystyle \min_{s\in[1,\infty)}h(s)} &:\; t\leq \min\argmin h\,,\\
			\hfill h(t) &:\; \text{otherwise,}
		\end{cases}
	\]
	where $\min\argmin h = \min\{s\in[1,\infty)\setvert h(s)=\min h\}$, cf.\ Figure \ref{fig:monotoneConvexEnvelopeExample}. In particular, if $h$ attains its minimum at $1$, then $\htilde=h$ and thus $\Cm h = Ch$. Note that if $h$ is not bounded below, then $\Cm h$ is not well defined (and neither is $RW$).
\end{remark}
\begin{remark}\label{remark:iwaniecResult}
	If $\Psi\col[1,\infty)\to\R$ is strictly monotone with sublinear growth, then both these properties hold for the function $h\col[1,\infty)\to\R$ with $\Psi(\K(F))=h(\frac{\lambda_1}{\lambda_2})\equalscolon W(F)$ as well, which implies
	\[
		QW = \Cm h = Ch \equiv h(1) = \Psi(1)\,.
	\]
	For this special case, we directly recover the earlier result \eqref{eq:iwaniecResult} originally due to \citet{astala2008elliptic}.
\end{remark}
\begin{remark}
The monotone-convex envelope of $h\col[1,\infty)\to\R$ can also be obtained by \enquote{reflecting} the graph of the function at $x=1$ and taking the classical convex envelope: if $\hhat\col\R\to\R$ denotes the extension of $h$ to $\R$ defined by
	\[
		\hhat(x)\colonequals\begin{cases}
		\hfill h(x) &:\; x>1\,,
		\\ h(1-x) &:\; x\leq 1\,,
		\end{cases}
	\]
	then $\Cm h = C \hhat |_{\R_{\geq 1}}$, cf.\ Figure~\ref{fig:monotoneConvexEnvelopeExample} and Appendix \ref{appendix:Dacorogna}.
\end{remark}
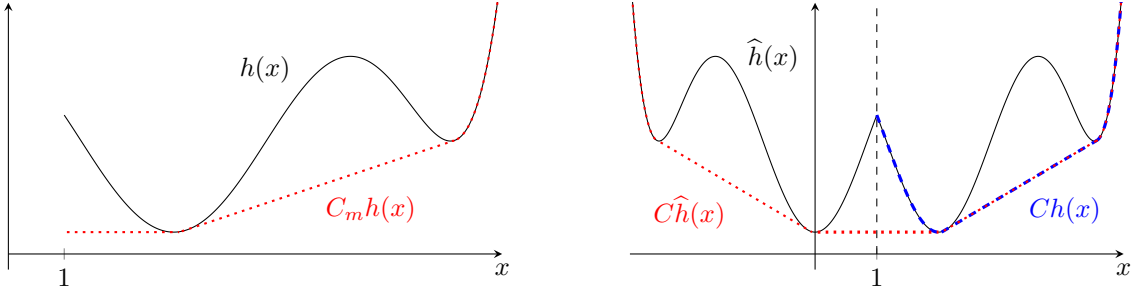
\begin{figure}[h!]
	\centering
	\def\yoffset{0.14}
   \begin{minipage}[b]{.49\linewidth}
      \centering
      \begin{tikzpicture}
		\begin{axis}[
        axis x line=middle,axis y line=middle,
        x label style={at={(current axis.right of origin)},anchor=north, below},
        xlabel=$x$, ylabel={},
        xmin=0.49, xmax=5,
        ymin=-0.1, ymax=1.61,
        width=1\linewidth,
        height=\graphRatio\linewidth,
        ytick=\empty,
        xtick={1}
        ]
        \addplot[black, smooth][domain=1:5, samples=\sample]{(x-2)^2-0.27*(x-2)^4+0.02*(x-2)^6+\yoffset} node[pos=.49, above left] {$h(x)$};
        \addplot[red, dash pattern=on 1pt off 2pt,shorten <=1,thick] coordinates {(1.0003,\yoffset) (2.05, \yoffset)};
        \addplot[red, dash pattern=on 1pt off 2pt,shorten <=1,thick] coordinates {(2.05,\yoffset) (4.55, {0.584+\yoffset})} node[midway, below right]{$\Cm h(x)$};
        \addplot[red, dash pattern=on 1pt off 2pt,shorten <=1,thick][domain=4.54:5, samples=\sample]{(x-2)^2-0.27*(x-2)^4+0.02*(x-2)^6+\yoffset};
        \end{axis}
	 	\end{tikzpicture}
    \end{minipage}
    \begin{minipage}[b]{.49\linewidth}
      \centering
      \begin{tikzpicture}
		\begin{axis}[
        axis x line=middle,axis y line=middle,
        x label style={at={(current axis.right of origin)},anchor=north, below},
        xlabel=$x$, ylabel={},
        xmin=-3, xmax=5,
        ymin=-0.1, ymax=1.61,
        width=1\linewidth,
        height=\graphRatio\linewidth,
        ytick=\empty,
        xtick={1}
        ]
        \addplot[black, dashed]coordinates{(1,0) (1,1.61)};
        \addplot[black, smooth][domain=1:5, samples=\sample]{(x-2)^2-0.27*(x-2)^4+0.02*(x-2)^6+\yoffset};
        \addplot[black, smooth][domain=-3:1, samples=\sample]{(x+0)^2-0.27*(x+0)^4+0.02*(x+0)^6+\yoffset} node[pos=.49, above right] {$\hhat(x)$};
        
        \addplot[blue, dash pattern=on 3pt off 3pt,very thick][domain=1:2, samples=\sample]{(x-2)^2-0.27*(x-2)^4+0.02*(x-2)^6+\yoffset};
        \addplot[blue, dash pattern=on 3pt off 3pt,very thick] coordinates {(2.05,\yoffset) (4.55, {0.584+\yoffset})} node[midway, below right]{$Ch(x)$};
        \addplot[blue, dash pattern=on 3pt off 3pt,very thick][domain=4.54:5, samples=\sample]{(x-2)^2-0.27*(x-2)^4+0.02*(x-2)^6+\yoffset};
        
        \addplot[red, dash pattern=on 1pt off 2pt,shorten <=1,very thick]coordinates{(-0,\yoffset) (2,\yoffset)};
        \addplot[red, dash pattern=on 1pt off 2pt,shorten <=1,thick] coordinates {(2.05,\yoffset) (4.55, {0.584+\yoffset})};
        \addplot[red, dash pattern=on 1pt off 2pt,shorten <=1,thick][domain=4.54:5, samples=\sample]{(x-2)^2-0.27*(x-2)^4+0.02*(x-2)^6+\yoffset};
        
        \addplot[red, dash pattern=on 1pt off 2pt,shorten <=1,thick][domain=-3:-2.54, samples=\sample]{(x+0)^2-0.27*(x+0)^4+0.02*(x+0)^6+\yoffset};        
        \addplot[red, dash pattern=on 1pt off 2pt,shorten <=1,thick] coordinates { (-2.54, {0.584+\yoffset})(-0.05,\yoffset)} node[midway, below left]{$C \hhat(x)$};

        \end{axis}
     	\end{tikzpicture}
	 \end{minipage}
  \caption{Left: Example of a monotone-convex envelope. Right: The monotone-convex envelope $\Cm h$ of $h\col[1,\infty)\to\R$ can be obtained by restricting the convex envelope $C\hhat$ of a suitably extension $\hhat\col\R\to\R$ of $h$ back to $[1,\infty)$.\label{fig:monotoneConvexEnvelopeExample}}
\end{figure}

\section{Specific relaxation examples and numerical simulations}
\label{section:applications}
Theorem \ref{theorem:mainResult} can be used to explicitly compute the quasiconvex envelope for a substantial class of functions. In the following, a number of explicit relaxation examples will be considered and some of our analytical results will be compared to numerical simulations.

\subsection{The deviatoric Hencky energy}
First, consider the (planar) \emph{deviatoric Hencky strain energy} \cite{Hencky1929,agn_neff2015geometry} $\WdH\col\GLp(2)\to\R$,
\begin{align*}
	\WdH(F) &= 2\.\norm{\dev_2\log U}^2 = 2\.\norm{\dev_2 \log\sqrt{F^TF}}=\left[\log\left(\frac{\norm{F}^2}{2\.\det F}+\sqrt{\frac{\norm{F}^4}{4\.(\det F)^2}-1}\right)\right]^2\\
	&=\left[\log\left(\K(F)+\sqrt{\K(F)^2-1}\right)\right]^2=\arccosh^2(\K(F))\,,
\end{align*}
where $\dev_n X \colonequals X-\frac{1}{n}\,\tr(X)\cdot\id$ is the deviatoric (trace-free) part of $X\in\Rnn$ and $\log U$ denotes the principal matrix logarithm of the right stretch tensor $U\colonequals \sqrt{F^T F}$. The energy $\WdH$ can be expressed as%
\[
	\WdH(F) = \log^2\left(\frac{\lambda_1}{\lambda_2}\right)\,.%
\]
Since the representing function $h\col[1,\infty)\to\R$ with $h(t)=\log^2(t)$ is monotone, we find
\[
	\Cm h(t) = Ch(t) = 0 \qquad\text{for all }\; t\in[1,\infty)
\]
and thus
\begin{align*}
	R\WdH = Q\WdH = P\WdH \equiv 0\,.
\end{align*}
Note that due to the sublinear growth of $h$ (or, equivalently, of the representation $\K\mapsto\arccosh^2(\K)$), this result can also be obtained by eq.\ \eqref{eq:iwaniecResult}, cf.\ Remark \ref{remark:iwaniecResult}.

\label{sectionContains:geodesicDistance}
Interestingly, the deviatoric Hencky strain energy itself is directly related to the conformal group $\CSO(n)$: Let $\dg(\cdot,\cdot)$ denote the \emph{geodesic distance} on the Lie group $\GLp(n)$ with respect to the canonical left-invariant Riemannian metric \cite{agn_martin2014minimal,Mielke2002}. Then the distance of $F\in\GLpn$ to the special orthogonal group $\SOn\subset\GLpn$ is given by \cite[Theorem 3.3]{agn_neff2015geometry}
\begin{equation}\label{eq:geodesicsMain}
	\dg^2(F,\SO(n)) = \min_{\Rtilde\in\SO(n)}\dg^2(F,\Rtilde) = \norm{\log U}^2\,.
\end{equation}
The deviatoric Hencky strain energy can therefore be characterized by the equality
\begin{align*}
	\dg^2(F,\CSO(n)) &= \min_{A\in\CSO(n)}\dg^2(F,A)= \min_{\substack{\Rtilde\in\SO(n)\\a\in(0,\infty)}}\dg^2(F,a\.\Rtilde)
	\;\overset{(*)}{=} \min_{a\in(0,\infty)\vphantom{\Rtilde}} \, \min_{\Rtilde\in\SO(n)}\dg^2\left(\frac{F}{a},\,\Rtilde\right)\notag\\
	\;&\overset{\mathclap{\eqref{eq:geodesicsMain}}}{=} \min_{a\in(0,\infty)} \, \Bignorm{\log\frac{U}{a}}^2
	= \min_{a\in(0,\infty)} \, \norm{(\log U) - \log(a)\.\id}^2
	= \norm{\dev_n\log U}^2\,,
\end{align*}
where $(*)$ holds due to the left-invariance of the metric.
\subsection{The squared logarithm of \texorpdfstring{$\K$}{K}}
Similarly, consider
\[
	\Wlog(F)
	=
	(\log\K)^2
	=
	\log^2\Big(\frac{1}{2}\Big(\frac{\lambda_1}{\lambda_2}+\frac{\lambda_2}{\lambda_1}\Big)\Big),
	\qquad\text{i.e.}\quad
	h(t)
	=
	\log^2\Big(\frac{1}{2}\Big(t+\frac{1}{t}\Big)\Big).
\]
Since $h$ is again monotone on $[1,\infty)$, we find
\[
	\Cm h(t) = Ch(t) = 0 \qquad\text{for all }\; t\in[1,\infty)
\]
and thus
\[
	R\Wlog = Q\Wlog = PW \equiv 0\,.
\]
\begin{figure}[h!]
  \centering
    \begin{minipage}[t]{.49\linewidth}
      \centering
      \begin{tikzpicture}
		\begin{axis}[
        axis x line=middle,axis y line=middle,
        x label style={at={(current axis.right of origin)},anchor=north, below},
        xlabel=$t$, ylabel=$h(t)$,
        xmin=-1, xmax=10,
        ymin=-0.4, ymax=4,
        width=1\linewidth,
        height=\graphRatio\linewidth,
        ytick=\empty,
        xtick={1},
        xticklabels={1}
        ]
        \addplot[black, smooth][domain=1:10, samples=\sample]{ln(x)^2};
        \end{axis}
     	\end{tikzpicture}
	 \end{minipage}
  	 \begin{minipage}[t]{.49\linewidth}
      \centering
      \begin{tikzpicture}
		\begin{axis}[
        axis x line=middle,axis y line=middle,
        x label style={at={(current axis.right of origin)},anchor=north, below},
        xlabel=$t$, ylabel=$h(t)$,
        xmin=-1, xmax=10,
        ymin=-0.4, ymax=4,
        width=1\linewidth,
        height=\graphRatio\linewidth,
        ytick=\empty,
        xtick={1},
        xticklabels={1}
        ]
        \addplot[black, smooth][domain=1:10, samples=\sample]{ln(0.5*(x+1/x))^2};
        \end{axis}
     	\end{tikzpicture}
	 \end{minipage}
  \caption{Left: Visualization of $\WdH(F) = \log^2\big(\frac{\lambda_1}{\lambda_2}\big)$ with $h(t)=\log^2(t)$. Right: Visualization of $\Wlog(F)=(\log\K)^2$ with $h(t)=\log^2\left(\frac{1}{2}\left(t+\frac{1}{t}\right)\right)$.}
\end{figure}
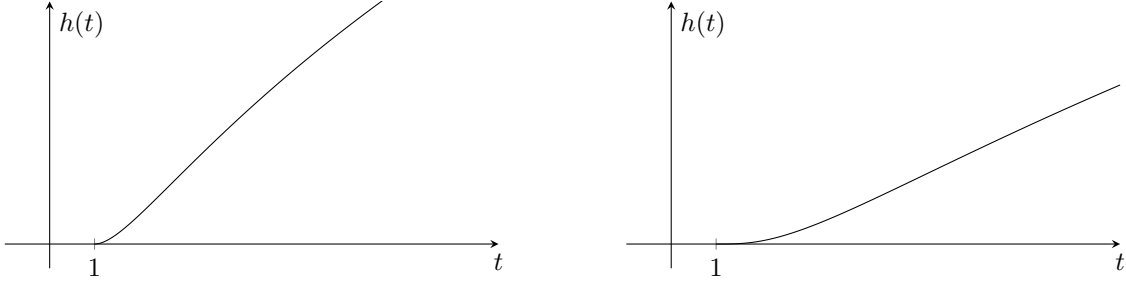
\subsection{The exponentiated Hencky energy}
Now, consider the \emph{exponentiated deviatoric Hencky energy} \cite{agn_neff2015exponentiatedI}
\[
	\WeH\col\GLp(2)\to\R\,,\quad \WeH(F) = e^{2k\norm{\dev_2\log U}^2}
\]
for some parameter $k>0$. It has previously been shown \cite{agn_neff2015exponentiatedII,agn_ghiba2015exponentiated,agn_martin2018non} that $\WeH$ is polyconvex (and thus quasiconvex) for $k\geq\frac18$. For any $0<k<\frac18$, we can explicitly compute the quasiconvex envelope: since
\[
	\WeH(F) = e^{k\log^2\big(\frac{\lambda_1}{\lambda_2}\big)}\,,
\]
and since the mapping $t\mapsto h(t)=e^{k\log^2(t)}$ is monotone increasing on $[1,\infty)$, we find
\[
	R\WeH(F) = Q\WeH(F) = P\WeH(F) = Ch\Big(\frac{\lambda_1}{\lambda_2}\Big)
\]
for all $F\in\GLp(2)$ with singular values $\lambda_1,\lambda_2$.

\medskip\noindent
In order to further investigate the behavior of this quasiconvex relaxation with finite element simulations, we choose the particular value $k = 0.11 < \frac{1}{8}$ and consider the quasiconvex envelope $QW(F)$ of
\begin{align*}
	W(F)&=h\left(\frac{\lambda_1}{\lambda_2}\right)=e^{0.11\left(\log\frac{\lambda_1}{\lambda_2}\right)^2}
	=e^{0.11\left[\arccosh\K(F)\right]^2}\,.
\end{align*}
Using Maxwell's equal area rule \cite[p.~319]{silhavy1997mechanics}, we numerically compute the monotone-convex envelope of $h$ up to five decimal digits:
\begin{align*}
	\Cm h(t)=Ch(t)=\begin{cases}
		\hfill h(t) &:\; \hphantom{2.65363}\mathllap{1}\leq t\leq 2.65363\,, \\
		0.872034+0.0898464\,t &:\; 2.65363<t<35.4998\,,\\
		\hfill h(t) &:\; 35.4998\leq t \,.
	\end{cases}
\end{align*}
This explicit representation allows us to determine the set of all $F\in\GLp(2)$ with $QW(F)<W(F)$, known as the \emph{binodal region} \cite{grabovsky2018rank,grabovsky2016legendre}. In particular, the microstructure energy gap (cf.~Figure~\ref{fig:energyGap}) between $h$ and $Ch$ is maximal at $\frac{\lambda_1}{\lambda_2}\approx 12.0186\equalscolon x_0$ with a value of $\Delta\approx0.0221558$. We therefore choose homogeneous Dirichlet boundary conditions given by
\begin{equation}
\label{eq:numeric_boundary_conditions}
	F_0=\begin{pmatrix}
		\sqrt{x_0} & 0\\ 0 & \frac{1}{\sqrt{x_0}}
	\end{pmatrix}=\begin{pmatrix}
		\sqrt{12.0186} & 0\\ 0 & \frac{1}{\sqrt{12.0186}}
	\end{pmatrix}\,,%
\end{equation}
such that $\det F_0=1$, for the finite element simulation. The energy level of the homogeneous solution is
\begin{align*}
	I(\varphi_0)=\int_{B_1(0)} W(F_0)\,\dx=\pi\.W(F_0)\approx 6.20155\,,
\end{align*}
whereas the infimum of the energy levels of the microstructure solutions is
\begin{align*}
	\inf I(\varphi)=\inf\left\{\int_{B_1(0)} W(\nabla\varphi)\,\dx\,,\quad\varphi|_{\partial B_1(0)}(x) =F_0\.x\right\}=\pi\.(W(F_0)-\Delta)\approx 6.13194\,.
\end{align*}

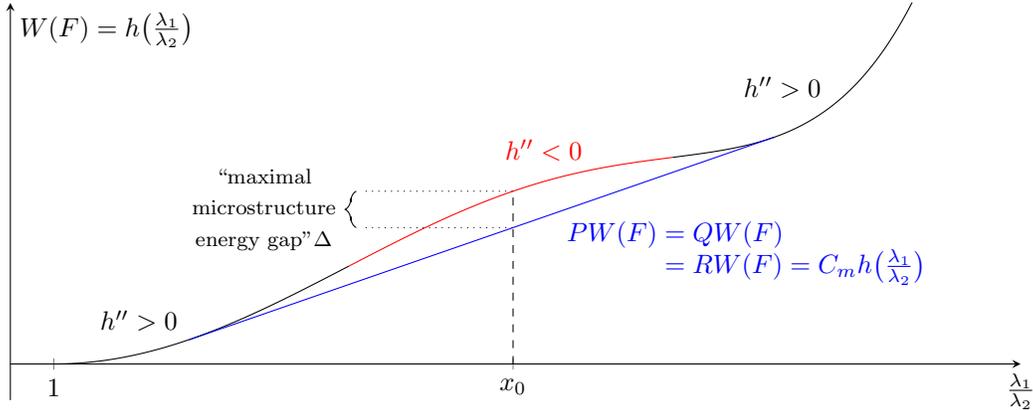
\begin{figure}[h!]
	\centering
  	\begin{tikzpicture}
		\begin{axis}[
		axis x line=middle,axis y line=middle,
		x label style={at={(current axis.right of origin)},anchor=north, below},
		xlabel=$\frac{\lambda_1}{\lambda_2}$, ylabel={$W(F)=h\bigl(\frac{\lambda_1}{\lambda_2}\bigr)$},
		xmin=.91, xmax=3,
		ymin=-0.1, ymax=1,
		width=.903\linewidth,
		height=.413\linewidth,
		ytick=\empty,
		xtick={1, 1.95},
		xticklabels={1, $x_0$}
		]
			\addplot[black, smooth][domain=1:1.61, samples=\sample]{0.9*(x-1)^2-0.5*(x-1)^4+0.1*(x-1)^6} node[pos=.42, above left] {$h''>0$};
			\addplot[red, smooth][domain=1.61:2.28, samples=\sample]{0.9*(x-1)^2-0.5*(x-1)^4+0.1*(x-1)^6} node[pos=.77, above left] {$h''<0$};
			\addplot[black, smooth][domain=2.28:3, samples=\sample]{0.9*(x-1)^2-0.5*(x-1)^4+0.1*(x-1)^6} node[pos=.21, above left] {$h''>0$};
			\addplot[blue, smooth] coordinates {(1.28, 0.067) (2.49, 0.628)} node[pos=.63, below right, align=left] {$PW(F)=QW(F)$\\$\hphantom{PW(F)}=RW(F)=\Cm h\bigl(\frac{\lambda_1}{\lambda_2}\bigr)$};
			\addplot[black,dashed] coordinates {(1.95, 0) (1.95, 0.4785)};
			\addplot[black,dotted] coordinates {(1.645, 0.4785) (1.95, 0.4785)} node[pos=0] (braceUpper) {};
			\addplot[black,dotted] coordinates {(1.645, 0.377636) (1.95, 0.377636)} node[pos=0] (braceLower) {};
		\end{axis}
		\draw[decorate,decoration={brace,amplitude=4.2pt,mirror}] (braceUpper.west) -- (braceLower.west) node [black,midway,left=4.2pt,align=center] {\footnotesize \textooquote maximal\\\footnotesize microstructure\\\footnotesize energy gap\textcoquote $\Delta$};
	\end{tikzpicture}
  \caption{\label{fig:energyGap}%
	  Visualization of the maximal microstructure energy gap $\Delta$ between $h$ and $\Cm h$ for an energy $W$ which is not convex with respect to $K(F)=\frac{\lambda_1}{\lambda_2}$, similar to the case $\WeH(F) = e^{k\.\log^2(\frac{\lambda_1}{\lambda_2})}$ for $k<\frac{1}{8}$.}
\end{figure}

Figure~\ref{fig:numerical_microstructure} shows two numerical simulations of the microstructure on triangle grids with different resolutions.  The illustration shows the reference configuration, colored according to the value of the determinant of the deformation gradient (plotting $\K$ instead results in similar images).  The energy level of the configuration on the left is $6.17149$ on a grid with $294\,912$ vertices. Repeating the computation on a grid with one additional step of uniform refinement leads to the configuration on the right, which has an energy level of $6.16216$.

Note that the values obtained for the energy level still differ significantly from the expected value of $6.13194$. It is unclear whether the discrepancy is solely due to insufficient mesh resolution; further numerical investigations on more performant hardware are planned for the future. The expected energy level was, however, obtained numerically using a modification of an algorithm by Bartels \cite{bartels2005reliable} for computing the polyconvex envelope.
\begin{figure}[h!]
 \begin{center}
  \includegraphics[width=0.45\textwidth]{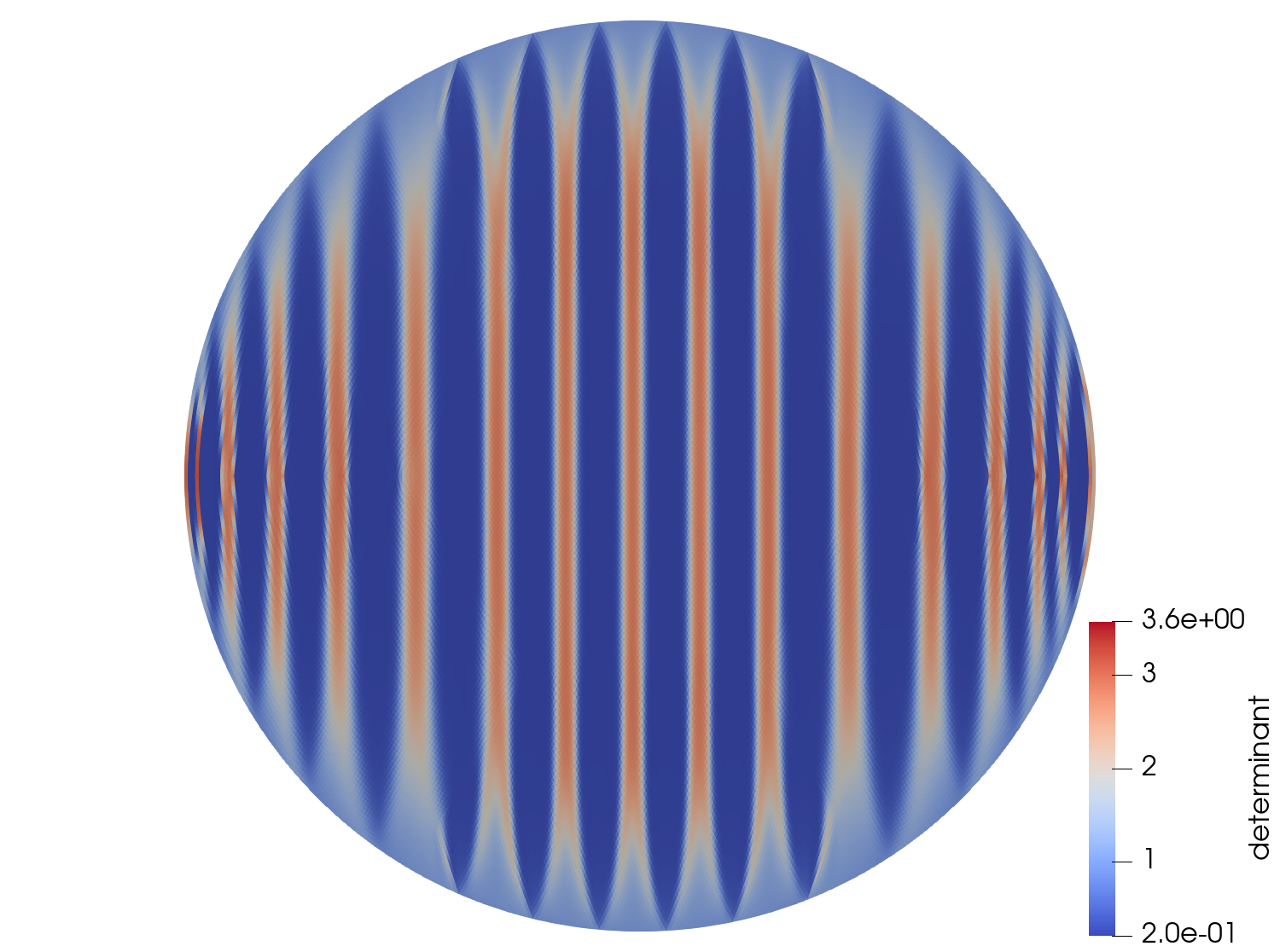}
  \includegraphics[width=0.45\textwidth]{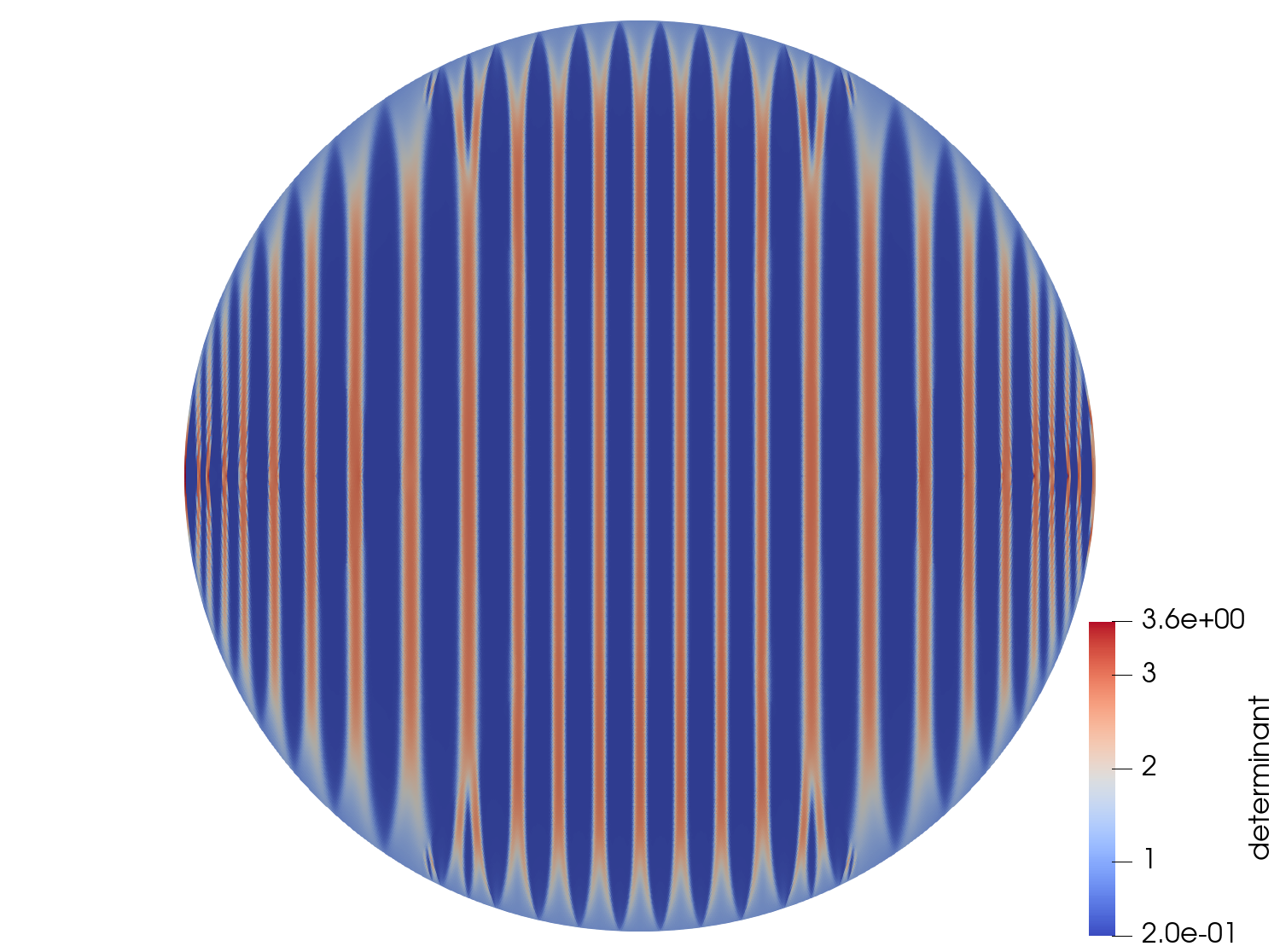}
 \end{center}
 \caption{Microstructure for the energy $W(F) = e^{0.11\left[\arccosh\K(F)\right]^2}$ with boundary conditions $F_0$ given by~\eqref{eq:numeric_boundary_conditions} for two different mesh resolutions. Although the number of oscillations (laminates) is mesh-dependent, macroscopic quantities like volume ratios are mesh-independent; these macroscopic features are predicted by $QW$. Left: $294\,912$ grid vertices, energy level of $6.17149$. Right: $1\,179\,648$ vertices, energy level of $6.16216$.}
 \label{fig:numerical_microstructure}
\end{figure}
\subsection{An energy function related to a result by Yan}
Lastly, we consider the energy function
\begin{align*}
	W(F)=\Psi_L(\K(F))=\cosh(\K(F)-L)-1=\cosh\left(\frac{1}{2}\left(\frac{\lambda_1}{\lambda_2}+\frac{\lambda_2}{\lambda_1}\right)-L\right)-1\,,
\end{align*}
which penalizes the deviation of the distortion $\K$ from a prescribed value $L\geq1$. According to Theorem~\ref{theorem:mainResult}, the quasiconvex envelope of $W$ is given by
\begin{equation}
	QW(F) = \begin{cases}
		\hfill 0 &:\; 1\leq\K(F) \leq L\,, \\
		W(F) &:\; L\leq\K(F)\,.
	\end{cases}\label{eq:exampleEnvelope}
\end{equation}
\begin{figure}[h!]
  \centering
  \tikzRemake
  \tikzsetnextfilename{hNotConvex}
    \begin{minipage}[t]{.49\linewidth}
      \centering
      \begin{tikzpicture}
		\begin{axis}[
        axis x line=middle,axis y line=middle,
        x label style={at={(current axis.right of origin)},anchor=north, below},
        xlabel=$\K$, ylabel=$\Psi_L(\K)$,
        xmin=-0.25, xmax=2.5,
        ymin=-0.2, ymax=2,
        width=1\linewidth,
        height=\graphRatio\linewidth,
        ytick=\empty,
        xtick={1,1.5},
        xticklabels={1,$L$}
        ]
        \addplot[black, smooth][domain=0.5:2.5, samples=\sample]{cosh(2*(x-1.5))-1};
        \end{axis}
     	\end{tikzpicture}
	 \end{minipage}
  	 \begin{minipage}[t]{.49\linewidth}
      \centering
      \begin{tikzpicture}
		\begin{axis}[
        axis x line=middle,axis y line=middle,
        x label style={at={(current axis.right of origin)},anchor=north, below},
        xlabel=$K$, ylabel=$h_l(K)$,
        xmin=0.21, xmax=4,
        ymin=-0.1, ymax=1,
        width=1\linewidth,
        height=\graphRatio\linewidth,
        ytick=\empty,
        xtick={1,2.618},
        xticklabels={1,$l$}
        ]
        \addplot[black, smooth][domain=0.25:4, samples=\sample*2]{cosh(2*(0.5*(x+1/x)-1.5))-1};
        \addplot[blue, dashed,thick][domain=2.618:4, samples=\sample]{cosh(2*(0.5*(x+1/x)-1.5))-1};
        \addplot[blue, dashed,thick] coordinates {(1,0.005) (2.618, 0.005)};
        \end{axis}
     	\end{tikzpicture}
	 \end{minipage}
  \caption{Visualization of $\Psi_L(\K)$, the corresponding representation $h_l(K)=\Psi_L\left(\frac{1}{2}\left(K+\frac{1}{K}\right)\right)$ and the monotone-convex envelope of the restriction of $h_l$ to $[1,\infty)$.}
\end{figure}
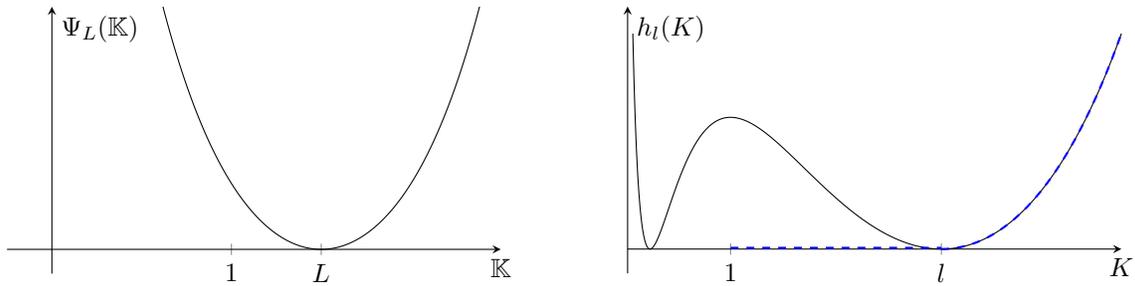

Again, we want to further investigate the microstructure induced by $W$ with numerical simulations on $\Omega=B_1(0)$.
For our calculations, we consider the case $L=2$. At $x_0=\frac{\lambda_1}{\lambda_2}=1$, the microstructure energy gap between between $h$ and $Ch$ is maximal with a value of $\Delta\approx 0.54308$, hence we use homogeneous Dirichlet boundary values with $F_0=\id$. The energy value of the homogeneous solution is
\begin{align*}
	I(\varphi_0)=\int_{B_1(0)} W(F_0)\,\dx=\pi\.W(F_0)\approx 1.70614,
\end{align*}
whereas the energy level of the microstructure solution should, in the limit, approach
\begin{align*}
	\inf I(\varphi)=\inf\left\{\int_{B_1(0)} W(\nabla\varphi)\,\dx\,,\quad\varphi|_{\partial B_1(0)}(x) = F_0x \right\}=\pi\.(W(F_0)-\Delta)=0\,.
\end{align*}

We again compute the microstructure using finite element simulations.
The microstructure exhibited by this example significantly differs from the previous one; in particular, no simple laminar structure can be observed at all (Figure~\ref{fig:numerical_microstructure_cosh_det}). As expected, we obtain deformations with $\K$ very close to the value~2 throughout the domain (Figure~\ref{fig:numerical_microstructure_cosh_K}). The energy levels obtained numerically are also very close to the expected value of $0$.  Specifically, for meshes with $294\,912$ and $1\,179\,648$ grid vertices, the obtained energy levels are $2.533\cdot 10^{-3}$ and $1.369\cdot 10^{-3}$, respectively.
\begin{figure}[h!]
 \begin{center}
  \begin{tikzpicture}
   \node at (0,0)
                  {\includegraphics[width=0.4\textwidth]{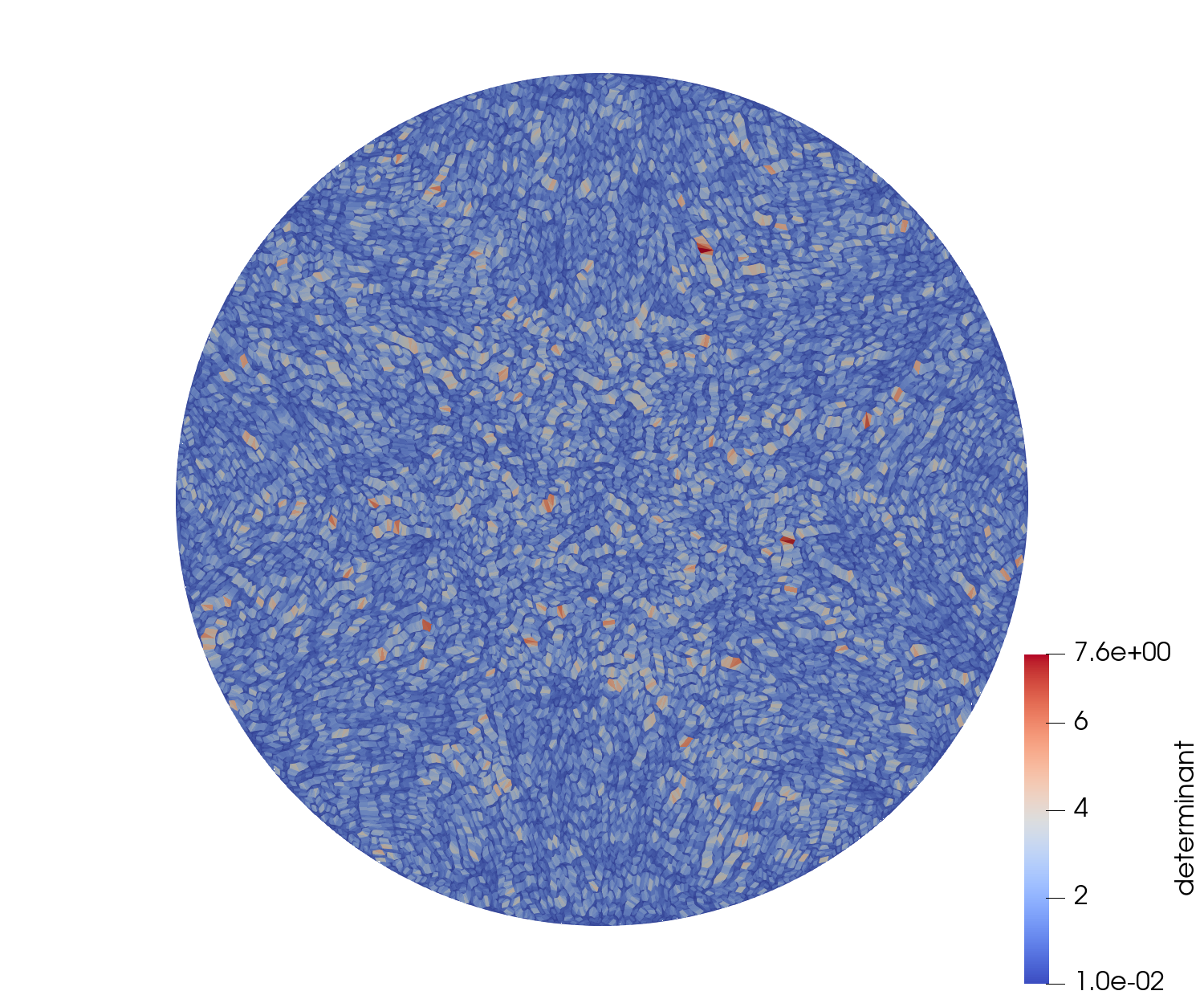}};

   \draw [thick] (1.78,0.8) rectangle ++(0.4,0.2);

   \node [draw, very thick, inner sep=0.5] at (8,1) (magnified)
                  {\includegraphics[width=0.4\textwidth]{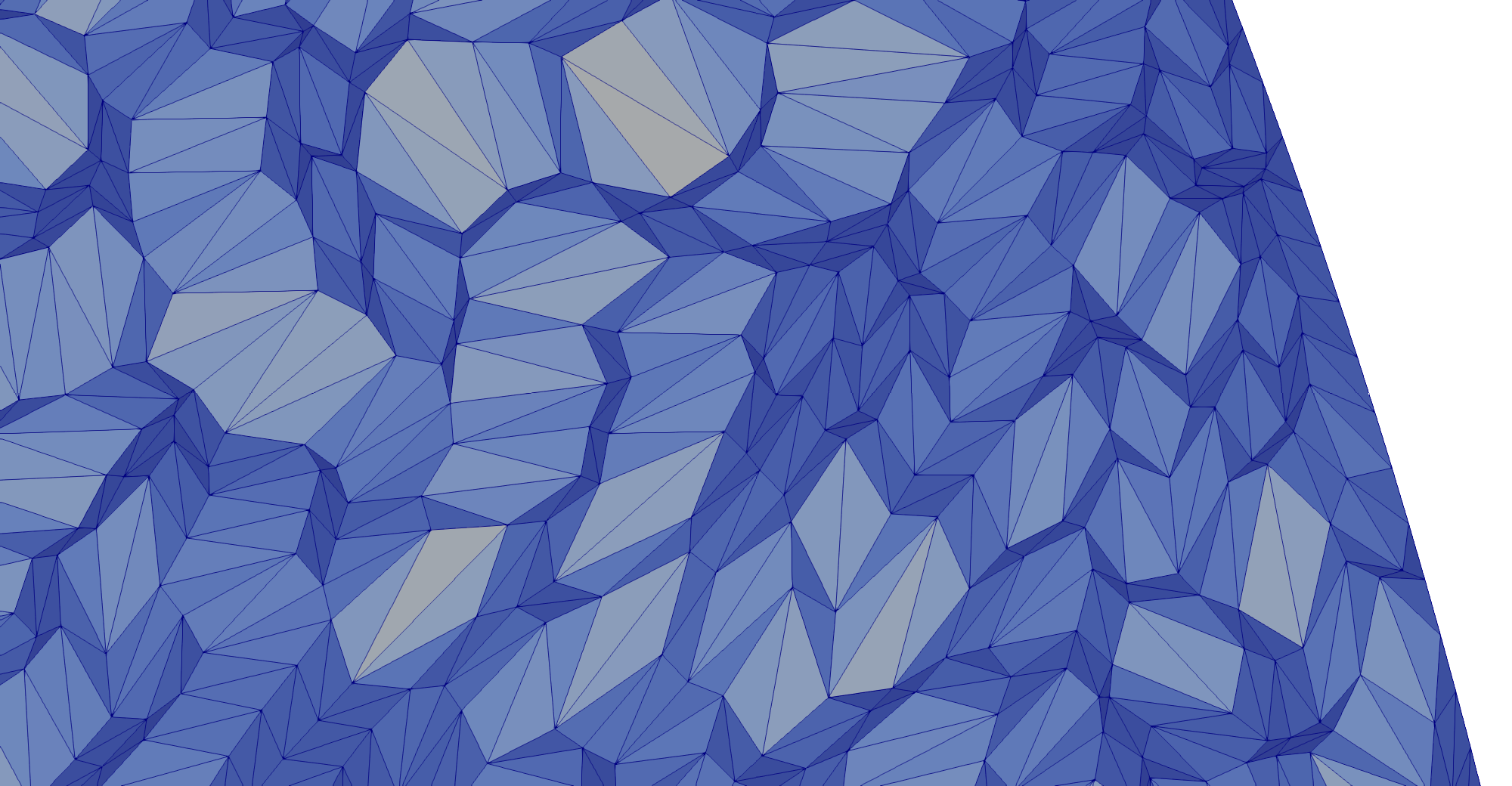}};

   \draw [thin,dashed] (1.78,0.8) -- (magnified.south west);
   \draw [thin,dashed] (1.78,1.0) -- (magnified.north west);
  \end{tikzpicture}
 \end{center}
 \vspace*{-1.4em}
 \caption{Microstructure for the energy $W(F) = \cosh(\K(F) -2) -1$ with boundary conditions $F_0 = \id$
   on a grid with $294\,912$ vertices (deformed configuration).  The coloring shows the distribution of $\det F$.}
 \label{fig:numerical_microstructure_cosh_det}
\end{figure}
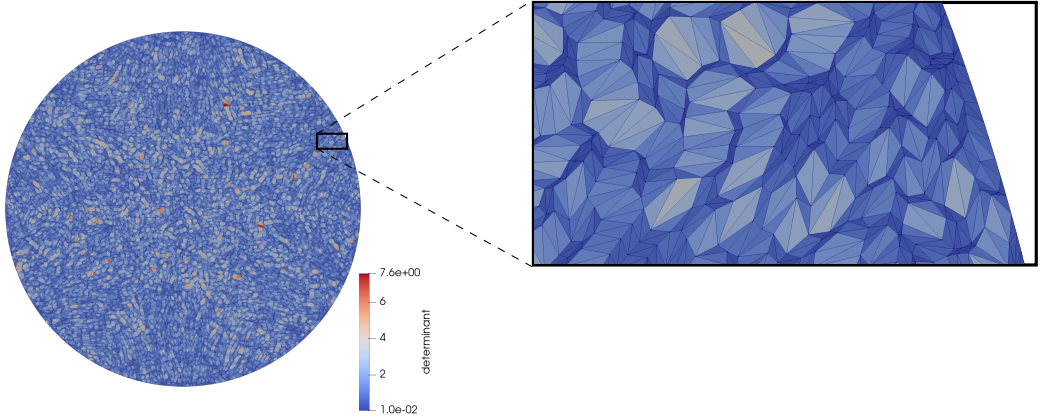
\begin{figure}[h!]
 \begin{center}
  \begin{tikzpicture}
   \node at (0,0)
                  {\includegraphics[width=0.4\textwidth]{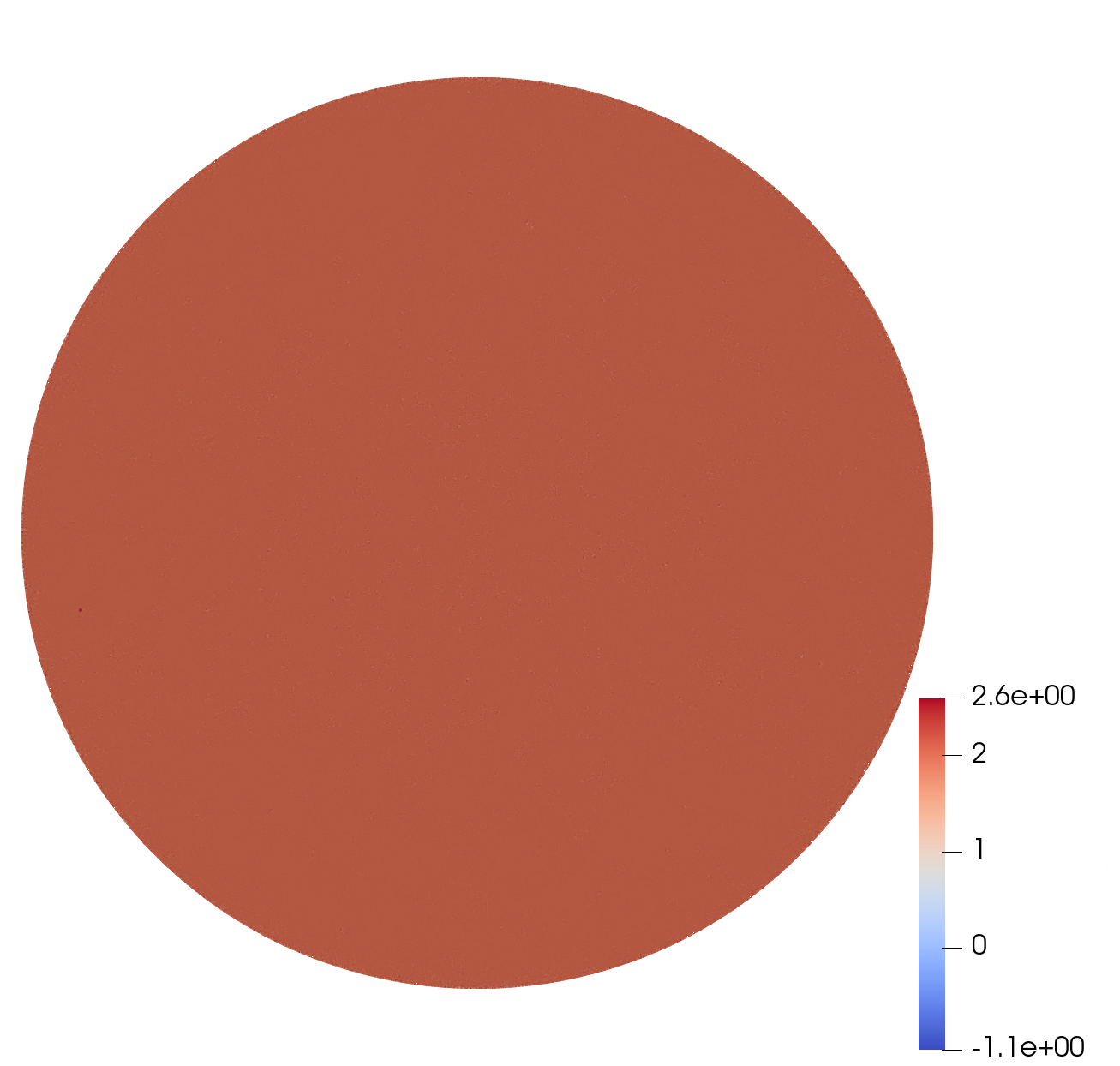}};

   \node at (2.75,-1.85) {$\K$};  %

   \draw [thick] (1.66,0.8) rectangle ++(0.4,0.2);

   \node [draw, very thick, inner sep=0.5] at (8,1) (magnified)
                  {\includegraphics[width=0.4\textwidth]{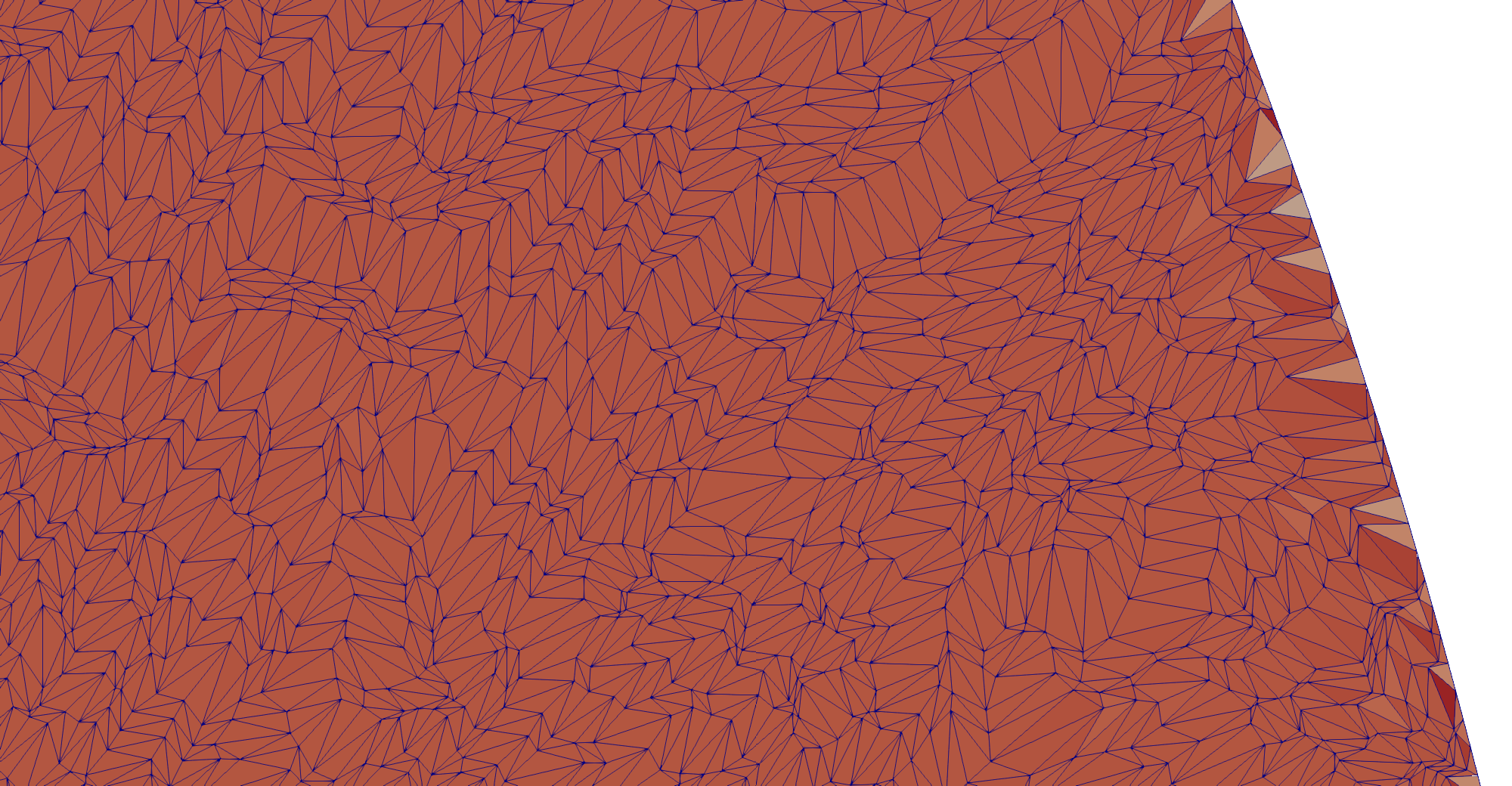}};

   \draw [thin,dashed] (1.66,0.8) -- (magnified.south west);
   \draw [thin,dashed] (1.66,1.0) -- (magnified.north west);
  \end{tikzpicture}
 \end{center}
 \vspace*{-1.4em}
 \caption{Microstructure for the energy $W(F) = \cosh(\K(F) -2) -1$ with boundary conditions $F_0 = \id$
   on a grid with $1\,179\,648$ vertices (deformed configuration).  The coloring shows the distribution of $\K$, which is essentially
   constant except near the boundary.}
 \label{fig:numerical_microstructure_cosh_K}
\end{figure}
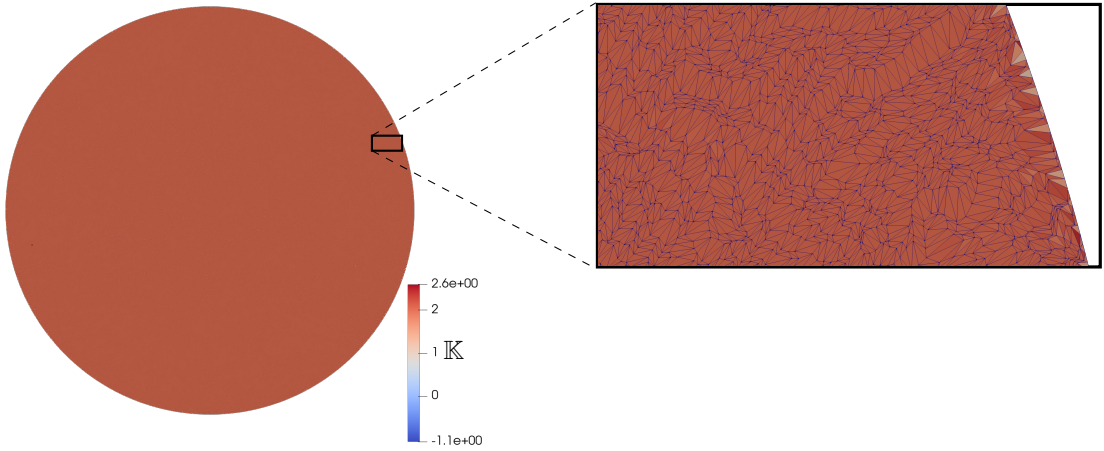

In the following, we will discuss a close connection of the quasiconvex envelope \eqref{eq:exampleEnvelope} and the observed microstructure to an earlier result by Yan \cite{yan2003baire,yan2001linear}, which implies that the microstructure should approximately satisfy $\K(\grad\varphi)=2$ almost everywhere on $\Omega$.

In two remarkable contributions \cite{yan2003baire,yan2001linear}, Yan considered the Dirichlet problem
\[
	\opnorm{\nabla\varphi}^n=l\.\det\nabla\varphi\qquad\text{a.e. in }\;\Omega\subset\R^n
\]
for an arbitrary number $l\geq 1$ under affine boundary conditions and obtained the following existence result.
\begin{theorem}[{\cite[Theorem 1.2]{yan2003baire}}]
\label{theorem:yan}
	Let $l\geq 1$. Given any affine map $x \mapsto F_0\.x+b$, the Dirichlet problem
	\begin{alignat*}{2}
		\norm{\nabla\varphi}_{\rm op}^n & =l\.\det\nabla\varphi & \qquad & \text{a.e.\ in }\; \Omega\\
		\varphi(x) & =F_0\.x+b  && \text{on }\; \partial \Omega
	\end{alignat*}
	is solvable in $W^{1,n}(\Omega;\R^n)$ if and only if $\norm{F_0}_{\rm op}^n\leq l\det F_0$.
\end{theorem}

Since in the two-dimensional case $\frac{\norm{\nabla\varphi}_{\rm op}^2}{\det\nabla\varphi} =\frac{\lambdamax}{\lambdamin}=K(\nabla\varphi)$, Yan's result can be stated in terms of the linear distortion $K$ for $n=2$.
\begin{corollary}\label{corollary:solvable}
	In the planar case $n=2$, for any affine map $x \mapsto F_0\.x+b$, the Dirichlet problem
	\begin{alignat*}{2}
		K(\nabla\varphi)&=l & \qquad & \text{a.e.\ in }\; \Omega\\
		\varphi(x) & = F_0\.x+b  && \text{on }\; \partial \Omega
	\end{alignat*}
	is solvable in $W^{1,2}(\Omega;\R^2)$ if and only if $K(F_0)\leq l$.
\end{corollary}
Furthermore, recalling that $\K=\frac{1}{2}\left(K+\frac{1}{K}\right)$ and letting $L=\frac{1}{2}\left(l+\frac{1}{l}\right)$, Corollary \ref{corollary:solvable} can equivalently be expressed in terms of the distortion $\K$.
\begin{corollary}\label{corollary:yan}
In the planar case $n=2$ for any affine map $x \mapsto F_0\.x+b$, the Dirichlet problem
	\begin{alignat*}{2}
		\tfrac{\norm{\nabla\varphi}^2}{2\.\det\nabla\varphi} =\K(\nabla\varphi)&=L & \qquad & \text{a.e.\ in }\; \Omega\\
		\varphi(x) & = F_0\.x+b  && \text{on }\; \partial \Omega
	\end{alignat*}
	is solvable in $W^{1,2}(\Omega;\R^2)$ if and only if $\frac{\norm{F_0}^2}{2\.\det F_0}=\K(F_0)\leq L$.
\end{corollary}
Using Corollary \ref{corollary:yan}, it is possible to obtain the relaxation result \eqref{eq:exampleEnvelope} by directly computing the quasiconvex envelope of $W(F)=\Psi_L(\K(F))=\cosh(\K(F)-L)-1$, i.e.
\begin{align*}
	QW(F_0)=\inf\left\{ \frac{1}{\abs{\Omega}} \int_{B_1(0)} \!\! W(\nabla\varphi)\,\dx\,,\;\varphi|_{\partial B_1(0)}=F_0\.x \right\}=\inf\left\{ \frac{1}{\abs{\Omega}} \int_{B_1(0)} \!\!\! \Psi_L(\K(\grad\varphi))\,\dx\,,\;\varphi|_{\partial B_1(0)}=F_0\.x \right\}\,.
\end{align*}
For $\K(F_0)=L$, the infimum value zero is already realized by the homogeneous solution. For $\K(F_0)<L$, although there is no homogeneous equilibrium solution, there exist a deformation $\widehat\varphi\in W^{1,2}(\Omega;\R^2)$ with $\widehat\varphi|_{\partial\Omega}=F_0\.x$ and $\K(\nabla\widehat\varphi)=L$ due to Corollary~\ref{corollary:yan}. Then $\Psi_L(\K(\nabla\widehat\varphi))=0$ and thus
\begin{align*}
	QW(F_0)=\begin{cases}
		\hfill 0 &:\; \K(F_0)\leq L\,, \\
		W(F_0) &:\; \K(F_0)\geq L\,,
	\end{cases}
\end{align*}
since $\Psi_L$ is monotone increasing and convex for $\K(F_0)\geq L$.
\section*{Acknowledgements}
We thank Sören Bartels (University of Freiburg) for providing his \textsc{Matlab} code to approximate the polyconvex envelope~\cite[p.\,285]{bartels2015numerical}, as well as Jörg Schröder (University of Duisburg-Essen) and Klaus Hackl (Ruhr University Bochum) for helpful discussions.
\par
\footnotesize
\captionsetup{font=footnotesize}
\section{References}
\printbibliography[heading=none]
\begin{appendix}
\section{The quasiconvex envelope for a class of conformal energies}\label{appendix:Dacorogna}
The concept of monotone-convex envelopes is directly connected to an earlier result by Dacorogna and Koshigoe \cite{dacorogna1993different}, who obtained an explicit relaxation result for a subclass of conformal energy functions.
\begin{lemma}[Proposition 5.1 \cite{dacorogna1993different}]
\label{lemma:dacorognaEnergy}
	Let $W\col\R^{2\times 2}\to\R$ be of the form
	\begin{equation}
	\label{eq:dacorognaConformalSpecialCase}
		W(F)\colonequals g(\sqrt{\norm{F}^2-2\.\det F})\,,\qquad g\col[0,\infty)\to\R\,.
	\end{equation}
	Define
	\[
		\gtilde\col\R\to\R\,,\qquad\gtilde(x)=
		\begin{cases}
		\hfill g(x) &:\; x>0\,,\\
		g(-x) &:\; x\leq 0\,.
		\end{cases}
	\]
	Then
	\begin{equation*}
		CW(F)=PW(F)=QW(F)=RW(F)=\gtilde^{**}\bigl(\sqrt{\norm{F}^2-2\.\det F}\.\bigr)\,,
	\end{equation*}
where $\gtilde^*$ is the Legendre-transformation of $\gtilde$ and $\gtilde^{**}=\left(\gtilde^*\right)^*$.
\end{lemma}
The same result can be found in \cite[Prop.~4.1]{vsilhavy2001rank}. Note that the convexity of the mapping $F\mapsto \gtilde^{**}(\sqrt{\norm{F}^2-2\.\det F})=\gtilde^{**}(\sqrt{(\lambda_1-\lambda_2)^2})$ follows directly \cite{ball1977constitutive} from the fact that $\gtilde^{**}$ is convex and non-decreasing on $[0,\infty)$.
Furthermore, if $g\geq0$, then $W$ of the form \eqref{eq:dacorognaConformalSpecialCase} is a conformal energy in the sense of Footnote \ref{footnote:conformalEnergy}.

If $g$ is continuous and bounded below, then based on \cite[Theorem 2.43]{Dacorogna08} it is easy to show that the monotone-convex envelope of $g$ is exactly the restriction of $\gtilde$ to $[0,\infty)$:
\begin{align*}
	\Cm g = (C\widetilde g)\big|_{[0,\infty)}\,,\qquad C\widetilde g=g^{**}\,.
\end{align*}

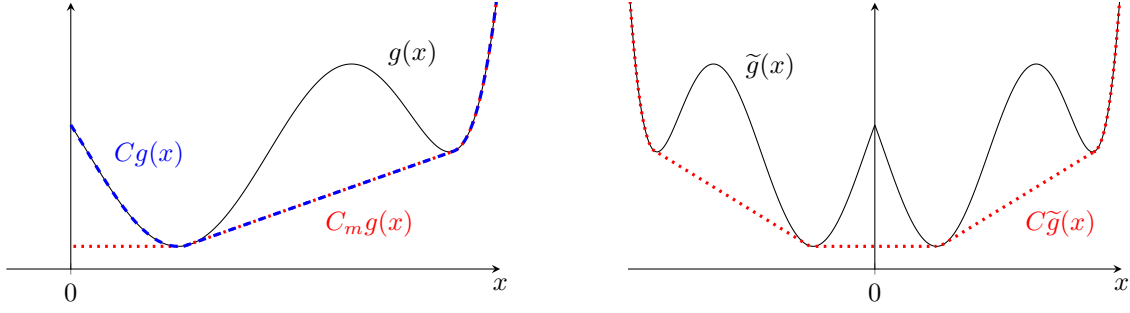
\begin{figure}[h!]
\centering
\def\yoffset{0.14}
	\tikzsetnextfilename{hNotConvex}
	\begin{minipage}[t]{.49\linewidth}
	\centering
		\begin{tikzpicture}
			\begin{axis}[
			axis x line=middle,axis y line=middle,
			x label style={at={(current axis.right of origin)},anchor=north, below},
			xlabel=$x$, ylabel={},
			xmin=-0.6, xmax=4,
			ymin=-0.0035, ymax=1.645,
			width=1\linewidth,
			height=\graphRatio\linewidth,
			ytick=\empty,
			xtick={0},
			hide obscured x ticks=false
			]
				\addplot[black, smooth][domain=0:4, samples=\sample]{(x-1)^2-0.27*(x-1)^4+0.02*(x-1)^6+\yoffset} node[pos=.63, above right] {$g(x)$};
				\addplot[red, dash pattern=on 1pt off 2pt,shorten <=1, very thick]coordinates{(0,\yoffset) (1,\yoffset)};
				\addplot[red, dash pattern=on 1pt off 2pt,shorten <=1, very thick] coordinates {(1.05,\yoffset) (3.55, {0.584+\yoffset})} node[pos=.49, below right] {$\Cm g(x)$};
				\addplot[red, dash pattern=on 1pt off 2pt,shorten <=1, very thick][domain=3.54:4, samples=\sample]{(x-1)^2-0.27*(x-1)^4+0.02*(x-1)^6+\yoffset};
				\addplot[blue, dash pattern=on 3pt off 3pt, very thick][domain=0:1, samples=\sample]{(x-1)^2-0.27*(x-1)^4+0.02*(x-1)^6+\yoffset} node[pos=.35, above right] {$Cg(x)$};
				\addplot[blue, dash pattern=on 3pt off 3pt, very thick] coordinates {(1.05,\yoffset) (3.55, {0.584+\yoffset})};
				\addplot[blue, dash pattern=on 3pt off 3pt, very thick][domain=3.54:4, samples=\sample]{(x-1)^2-0.27*(x-1)^4+0.02*(x-1)^6+\yoffset};
			\end{axis}
		\end{tikzpicture}
	\end{minipage}
	\begin{minipage}[t]{.49\linewidth}
	\centering
		\begin{tikzpicture}
			\begin{axis}[
			axis x line=middle,axis y line=middle,
			x label style={at={(current axis.right of origin)},anchor=north, below},
			xlabel=$x$, ylabel={},
			xmin=-4, xmax=4,
			ymin=-0.0035, ymax=1.645,
			width=1\linewidth,
			height=\graphRatio\linewidth,
			ytick=\empty,
			xtick={0},
			hide obscured x ticks=false
			]
				\addplot[black, smooth][domain=-4:0, samples=\sample]{(x+1)^2-0.27*(x+1)^4+0.02*(x+1)^6+\yoffset} node[pos=.49, above right] {$\gtilde(x)$};
				\addplot[black, smooth][domain=0:4, samples=\sample]{(x-1)^2-0.27*(x-1)^4+0.02*(x-1)^6+\yoffset};
				\addplot[red, dash pattern=on 1pt off 2pt,shorten <=1, very thick][domain=-4:-3.54, samples=\sample]{(x+1)^2-0.27*(x+1)^4+0.02*(x+1)^6+\yoffset};
				\addplot[red, dash pattern=on 1pt off 2pt,shorten <=1, very thick] coordinates { (-3.54, {0.584+\yoffset})(-1.05,\yoffset)};
				\addplot[red, dash pattern=on 1pt off 2pt,shorten <=1,very  thick]coordinates{(-1,\yoffset) (1,\yoffset)};
				\addplot[red, dash pattern=on 1pt off 2pt,shorten <=1, very thick] coordinates {(1.05,\yoffset) (3.55, {0.584+\yoffset})} node[pos=.49, below right]{$C\gtilde(x)$};
				\addplot[red, dash pattern=on 1pt off 2pt,shorten <=1, very thick][domain=3.54:4, samples=\sample]{(x-1)^2-0.27*(x-1)^4+0.02*(x-1)^6+\yoffset};
			\end{axis}
		\end{tikzpicture}
	\end{minipage}
	\caption{\footnotesize The monotone-convex envelope $\Cm g$ of $g\col[0,\infty)\to\R$ can be obtained via the convex envelope $\Cm g$ of the even extension $\gtilde$ of $g$.}
\end{figure}

Similar to the geodesic distance considered in Section \ref{sectionContains:geodesicDistance}, the expression $\sqrt{\norm{F}^2-2\.\det F}$ can be characterized as a measure of distance to the conformal group:\footnote{Note that the Euclidean distance can be considered a linearization of the geodesic distance and, unlike the latter, does not take into account the Lie group structure of either $\GLp(2)$ or $\CSO(2)$. For a detailed discussion of the relation between these distance measures and their applicability to the deformation gradient in nonlinear mechanics, see \cite{agn_neff2015geometry}.} since the closure $\CSO(2)\cup\{0\}$ of $\CSO(2)$ is a linear subspace\footnote{More generally \cite[p.24]{vsilhavy2004energy}, the set $[0,\infty)\cdot\SO(n)$ is convex for $n\geq1$.} of $\R^{2\times2}$ with an orthonormal basis given by
\[
	A_1 = \frac{\sqrt{2}}{2}\.\matr{1&0\\0&1}\,,\qquad A_2 = \frac{\sqrt{2}}{2}\.\matr{0&1\\-1&0}\,,
\]
thus
\begin{align*}
	\de^2(F,\CSO(2)) \colonequals& \inf_{A\in\CSO(2)} \norm{F-A}^2\\
	=&\; \norm{F}^2 - (\iprod{F,A_1}^2+\iprod{F,A_2}^2)
	= \norm{F}^2 - \frac12\.((F_{11}+F_{22})^2 + (F_{12}-F_{21})^2)\\
	=&\; \norm{F}^2 - \frac12\.(F_{11}^2+F_{22}^2+F_{12}^2+F_{21}^2 + 2\.(F_{11}\.F_{22}-F_{12}\.F_{21}))
	\;=\; \frac12\.(\norm{F}^2 - 2\.\det F)\,,
\end{align*}
where $\iprod{\cdot,\cdot}$ denotes the canonical inner product on $\R^{2\times2}$. Therefore, the energy functions considered in Lemma \ref{lemma:dacorognaEnergy} depend only on the Euclidean distance of $F$ to $\CSO(2)$.

\section{Connections to the Grötzsch problem}
\label{section:groetzsch}
Proposition~\ref{proposition:convexityCharacterization} negatively answers a conjecture by Adamowicz \cite[Conjecture~1]{adamowicz2007grotzsch}, which (in the two-dimensional case) states that if a conformal energy $W\col\GLp(2)\to\R$ with $W(F)=\Psi(\K(F))$ is polyconvex, then $\Psi$ is non-decreasing and convex. A direct counterexample is given by $W(F)=\frac{\lambdamax}{\lambdamin}$, which is polyconvex due to criterion~v) in Proposition \ref{proposition:convexityCharacterization} with $h(t)=t$, but the representation $W(F)=\Psi(\K(F))=e^{\arccosh(\K(F))}$ is not convex with respect to $\K(F)$.

Furthermore, criterion~iv) in Proposition~\ref{proposition:convexityCharacterization} reveals a direct connection between the so-called \emph{Grötzsch property} and quasiconvexity in the two-dimensional case.

\begin{definition}[\cite{adamowicz2007grotzsch}]
\label{definition:groetzschProperty}
	Let $W\col\GLp(n)\to\R$ be conformally invariant. Then $W$ satisfies the \emph{Grötzsch property} if for every $\Q=[0,a_1]\times\cdots\times[0,a_n]\subset\R^n$ and every $\Q'=[0,a'_1]\times\cdots\times[0,a'_n]\subset\R^n$, the functional
	\[
		I\col\mathcal{A}\to\R\,,\quad I(\varphi) = \int_{\Q} W(\grad\varphi)\,\dx
	\]
	attains its minimum at the affine mapping $\varphi\col\Q\to\Q'$, $\varphi(x)=(\frac{a'_1}{a_1}x_1,\dotsc,\frac{a'_n}{a_n}x_n)$; here, the set $\mathcal{A}$ of admissible functions consists of all $\varphi\in W^{1,p}_{\mathrm{loc}}(\Q;\Q'),\.p\geq n$ with $\det\grad\varphi>0$ that satisfy the \emph{Grötzsch boundary conditions}, i.e.\ map each $(n-1)$--dimensional face of $\Q$ to the corresponding face of $\Q'$.
\end{definition}

Note that the boundary condition imposed in Definition \ref{definition:groetzschProperty} does not require the admissible mappings to be affine at the boundary, since each of the faces can be mapped to the corresponding ones in an arbitrary (possibly non-linear) manner.

In the two-dimensional case, the representation of the energy in terms of the singular values allows us to infer the quasiconvexity from the Grötzsch property in a particularly straightforward way.

\begin{proposition}
	Let $W\col\GLp(2)\to\R$ be conformally invariant and satisfy the Grötzsch property for all $\Q,\Q'$. Then $W$ is polyconvex.
\end{proposition}
\begin{proof}
	Assume that $W$ is not polyconvex. Then $g\col(0,\infty)^2\to\R$ with $W(F)=g(\lambda_1,\lambda_2)$ is not separately convex according to criterion iv) in Proposition \ref{proposition:convexityCharacterization}. Therefore, there exist $\lambda_1,\lambdahat_1,\lambda_2\in(0,\infty)$ and $t\in(0,1)$ such that
	\[
		tg(\lambda_1,\lambda_2)+(1-t)g(\lambdahat_1,\lambda_2)<g(t\lambda_1+(1-t)\lambdahat_1,\lambda_2)\,.
	\]
	Now, let $\Q=[0,1]^2$ and $\Q'=[0,t\lambda_1+(1-t)\lambdahat_1]$, and define $\varphi\col\Q\to\Q'$ by
	\begin{align*}
		\varphi(x) \colonequals
		\begin{cases}
			\hfill\matr{\lambda_1\.x_1\\\lambda_2\.x_2} &:\; x_1\leq t\,,\\[2.1em]
			\matr{\lambdahat_1\.x_1+t(\lambda_1-\lambdahat_1)\\\lambda_2 x_2} &:\; x_1>t\,.
		\end{cases}
	\end{align*}
	Then $\varphi$ satisfies the Grötzsch boundary conditions, $\varphi\in W^{1,p}(\Q;\Q')$ for all $p\geq1$ and
	\begin{align*}
			\int_\Q W(\grad\varphi)\,\dx &= \int_{[0,t]\times[0,1]} W(\diag(\lambda_1,\lambda_2))\,\dx + \int_{[t,1]\times[0,1]} W(\diag(\lambdahat_1,\lambda_2))\,\dx\\
			&= \int_{[0,t]\times[0,1]} g(\lambda_1,\lambda_2)\,\dx + \int_{[t,1]\times[0,1]} g(\lambdahat_1,\lambda_2)\,\dx= tg(\lambda_1,\lambda_2)+(1-t)g(\lambdahat_1,\lambda_2)\\
			&<g(t\lambda_1+(1-t)\lambdahat_1,\lambda_2)
			\;=\; W(F_0) = W(F_0)\cdot\abs{\Q}\,,
	\end{align*}
	where $F_0=\diag(t\lambda_1+(1-t)\lambdahat_1,\lambda_2)$ is the boundary-compatible linear mapping from $\Q$ to $\Q'$. Therefore, $W$ does not satisfy the Grötzsch condition.
\end{proof}
\section{The convex envelope of conformally invariant planar energies}\label{appendix:convexity}

The quasiconvex envelopes computed in Section \ref{section:applications} are, in general, not convex, i.e.\ $QW(F)>CW(F)$ for some $F\in\GLp(2)$. In fact, the following explicit computation shows that the convex envelope of any conformally invariant energy is necessarily constant.

Recall that the convex envelope $CW$ of an energy $W\col M\to\R$ with a non-convex domain $M\subset\Rnn$ (e.g.\ $M=\GLp(2)$) is defined as the restriction $C\Wtilde|_M$ of the convex envelope $C\Wtilde$ of the function
\[
	\Wtilde\col\conv(M)\to\R\cup\{+\infty\}\,,\quad \Wtilde(F)=
	\begin{cases}
		W(F) &:\; F\in M\,,\\
		\hfill+\infty &:\; F\notin M
	\end{cases}
\]
to $M$, where $\conv(M)$ denotes the convex hull of the set $M$. Note that $\Wtilde$ can be further extended to a convex function $\What$ on $\Rnn$ by setting $\What(F)=+\infty$ for all $F\notin\conv(M)$.

\begin{proposition}
\label{prop:convexEnvelope}
	Let $W\col\GLp(2)\to\R$ be conformally invariant and bounded below. Then 
	\[
		CW(F)=\inf\left\{W(\Ftilde)\;|\;\Ftilde\in\GLp(2)\right\}
	\]
	for all $F\in\GLp(2)$.
\end{proposition}
\begin{proof}
	We only need to show that $CW$ is constant on $\GLp(2)$. First, observe that the convex envelope of $W$ is conformally invariant.\footnote{%
		The proof of the bi-$\SO(2)$-invariance of $CW$ given by Buttazzo et al.\ \cite[Therem 3.1]{buttazzo1994envelopes} can easily be adapted to include the scaling invariance.%
	}
	By the definition of convexity on $\GLp(2)$ employed here, $CW$ must be the restriction of a convex function $\What\col\R^{2\times2}\to\R$ to $\GLp(2)$. Let $b\colonequals \What(0)$. Then for all $F\in\GLp(2)$ and $t\in[-1,1]$, we find
	\[
		\What(tF) =
		\begin{cases}
			CW(tF) = CW(F) &:\; t\neq0\,,\\
			\hfill b &:\; t=0\,,
		\end{cases}
	\]
	thus $CW(F)=b$ due to the convexity of $\What$.%
\end{proof}

As indicated in Section \ref{section:convexityProtperties}, analytical methods for finding generalized convex envelopes have often been based on the observation that $RW=CW$ for certain classes of energy functions $W$ and the subsequent computation of the classical convex envelope $CW$; for example, this method is applicable to the St.\,Venant--Kirchhoff energy function \cite{ledret1995quasiconvex} $\WSVK(F)=\frac{\mu}{4}\.\norm{F^TF-\id}^2+\frac{\lambda}{8}\left(\tr(F^TF-\id)\right)^2$.

One of the most frequently cited examples of an isotropic and objective energy function $W$ with $RW=QW=PW\neq CW$ is the example of Kohn and Strang \cite{kohn1983explicit,kohn1986optimal}, where, in the $\R^{2\times 2}$-case \cite{zhang2002elementary,dovsly1997remark},
\begin{align*}
	W(F)=
	\begin{cases}
		1+\norm{F}^2 &:\; F \neq 0\,,\\
		\hfill 0 &:\; F=0\,,
	\end{cases}
	\qquad\text{with}\qquad
	CW(F)&=
	\begin{cases}
		W(F) &:\; \norm{F}\geq 1\,,\\
		\hfill 2\.\norm{F} &:\; \norm{F}< 1\,,
	\end{cases}
	\\
	\text{but}\qquad
	QW(F)&=
	\begin{cases}
		\hfill W(F) &:\; \norm{F}+2\.\det F\geq 1\,,\\
		2\.\sqrt{\norm{F}^2+2\.\det F}-2\.\det F &:\; \norm{F}+2\.\det F< 1\,.
	\end{cases}
\end{align*}
Furthermore, if $W\col\GLp(n)\to\R$ is a volumetric energy function of the form $W(F)=f(\det F)$ with $f\col(0,\infty)\to\R$, then $RW(F)=QW(F)=PW(F)=Cf(\det F)$ and, in general, $CW(F)<QW(F)$, see \cite[Theorem 6.24]{Dacorogna08}.
\end{appendix}
\end{document}